\RequirePackage{fix-cm}
\documentclass[smallextended]{svjour3}       
\smartqed  

\usepackage{multirow,mathtools}
\usepackage{amsmath,amssymb}
\usepackage{amsfonts}
\usepackage{url}
\usepackage[noend]{algpseudocode}
\usepackage{algorithm,algorithmicx}
\usepackage{color}
\usepackage[pointedenum]{paralist}
\usepackage{listings}
\usepackage{cases}
\usepackage{tikz}
\usetikzlibrary{matrix,decorations.pathreplacing}
\usepackage{graphicx}
\usepackage{subcaption}
\usepackage{grffile}
\usepackage[breaklinks=true,pdfstartview=FitH]{hyperref}

\newcommand{\modify}[1]{{{#1}}}
\newcommand{\anonymous}[2]{{#2}}

\def\noprint#1{}

\newtheorem{assumption}{Assumption}

\newcommand{\lambdamin}{\lambda_{\mbox{\rm\scriptsize{min}}}}

\newcommand{\bfone}{\mathbf{1}}

\newcommand{\R}{\mathbb{R}}

\newcommand{\E}{\mathbb{E}}

\def\sjwcomment#1{\footnote{SJW: #1}}

\newcommand{\TheTitle}{Inexact Successive Quadratic Approximation for Regularized Optimization}
\newcommand{\RunningTitle}{Successive Quadratic Approximation for Regularized Optimization}

\titlerunning{\RunningTitle}

\title{{\TheTitle}
	\thanks{This work was supported by NSF awards 1447449, 1628384,
		1634579, and 1740707;
		Subcontracts and 3F-30222 and 8F-30039 from Argonne National
		Laboratory; and
	Award N660011824020 from the DARPA Lagrange Program.}
}

\author{
Ching-pei Lee
\and
Stephen~J. Wright
}

\institute{Ching-pei Lee \at
	\email{ching-pei@cs.wisc.edu}\\
	Stephen~J. Wright \at
	\email{swright@cs.wisc.edu}\\
	Computer Sciences Department and Wisconsin Institute for
	Discovery, University
	of Wisconsin-Madison, Madison, WI, USA
}
\date{received: March 2018 / accepted: January 2019}
\newcommand{\citet}[1]{\cite{#1}}
\newcommand{\citep}[1]{\cite{#1}}
\journalname{Computational Optimization and Applications}
\begin{document}

\maketitle

\begin{abstract}
Successive quadratic approximations, or second-order proximal methods,
are useful for minimizing functions that are a sum of a smooth part
and a convex, possibly nonsmooth part that promotes regularization.
Most analyses of iteration complexity focus on the special case of
proximal gradient method, or accelerated variants thereof.  There have
been only a few studies of methods that use a second-order
approximation to the smooth part, due in part to the difficulty of
obtaining closed-form solutions to the subproblems at each
iteration. In fact, iterative algorithms may need to be used to find
inexact solutions to these subproblems.
In this work, we present global analysis of the iteration complexity
of inexact successive quadratic approximation methods, showing that an
inexact solution of the subproblem that is within a fixed
multiplicative precision of optimality suffices to guarantee the same
order of convergence rate as the exact version, with complexity
related in an intuitive way to the measure of inexactness. Our result allows
flexible choices of the second-order term, including Newton and
quasi-Newton choices, and does not necessarily require increasing
precision of the subproblem solution on later iterations.  For
problems exhibiting a property related to strong convexity, the
algorithms converge at global linear rates. For general convex
problems, the convergence rate is linear in early stages, while the
overall rate is $O(1/k)$. For nonconvex problems, a first-order
optimality criterion converges to zero at a rate of $O(1/\sqrt{k})$.
\keywords{ Convex optimization \and Nonconvex optimization \and
  Regularized optimization \and Variable
  metric \and Proximal method \and Second-order approximation \and
  Inexact method}
\end{abstract}


\section{Introduction} \label{sec:intro}
We consider  the following regularized optimization problem:
\begin{equation} \label{eq:f}
	\min_{x} \, F(x) \coloneqq f(x) + \psi(x),
\end{equation}
where $f:\R^n \to \R$ is $L$-Lipschitz-continuously differentiable,
and $\psi:\R^n \to \R$ is convex, extended-valued, proper, and closed,
but might be nondifferentiable.  Moreover, we assume that $F$ is
lower-bounded and the solution set $\Omega$ of \eqref{eq:f} is
non-empty. Unlike the many other works on this topic, we focus on the
case in which $\psi$ does {\em not} necessarily have a simple structure,
such as (block) separability, which allows a prox-operator to be
calculated economically, often in closed form. Rather, we assume that
subproblems that involve $\psi$ explicitly are solved inexactly, by an
iterative process.


Problems of the form \eqref{eq:f} arise in many contexts. The function
$\psi$ could be an indicator function for a trust region or a convex
feasible set. It could be a multiple of an $\ell_1$ norm or a
sum-of-$\ell_2$ norms. It could be the nuclear norm for a matrix
variable, or the sum of absolute values of the elements of a
matrix. It could be a smooth convex function, such as $\| \cdot\|_2^2$
or the squared Frobenius norm of a matrix. Finally, it could be a
combination of several of these elements, as happens when different
types of structure are present in the solution. In some of these
situations, the prox-operator involving $\psi$ is expensive to calculate
exactly.




We consider algorithms that generate a sequence $\{ x^k
\}_{k=0,1,\dotsc}$ from some starting point $x^0$, and solve the
following subproblem inexactly at each iteration, for some symmetric
matrix $H_k$:
\begin{equation}
\label{eq:quadratic}
\arg\min_{d\in \R^n} \, Q^{x^k}_{H_k}\left( d \right) \coloneqq
\nabla f\left( x^k \right)^Td + \frac12 d^T{H_k}d + \psi\left( x^k + d
\right) - \psi\left( x^k \right).
\end{equation}
We abbreviate the objective in \eqref{eq:quadratic} as $Q_k(\cdot)$
(or as $Q(\cdot)$ when we focus on the inner workings of iteration $k$).
In some results, we allow $H_k$ to have zero or negative eigenvalues,
provided that $Q_k$ itself is strongly convex. (Strong convexity in
$\psi$ may overcome any lack of strong convexity in the quadratic part of
\eqref{eq:quadratic}.)

In the special case of the proximal-gradient algorithm
\citep{ComW05a,WriNF08a}, where $H_k$ is a positive multiple of the
identity, the subproblem \eqref{eq:quadratic} can often be solved
cheaply, particularly when $\psi$ is (block) separable, by means of a
prox-operator involving $\psi$. For more general choices of $H_k$, or for
more complicated regularization functions $\psi$, it may make sense to
solve \eqref{eq:quadratic} by an iterative process, such as
accelerated proximal gradient or coordinate descent.  Since it may be
too expensive to run this iterative process to obtain a high-accuracy
solution of \eqref{eq:quadratic}, we consider the possibility of an
inexact solution. In this paper, we assume that the inexact solution
satisfies the following condition, for some constant $\eta \in [0,1)$:
\begin{equation}
Q\left(d\right) - Q^*  \le \eta\left(Q\left(0\right) - Q^*\right) 
\quad \Leftrightarrow \quad Q(d) \le (1-\eta) Q^*,
\label{eq:approx}
\end{equation}
where $Q^* \coloneqq \inf_d Q(d)$ and $Q(0)=0$. The value $\eta=0$
corresponds to exact solution of \eqref{eq:quadratic}. Other values
$\eta \in (0,1)$ indicate solutions that are inexact to within a {\em
  multiplicative} constant.

The condition \eqref{eq:approx} is studied in
\cite[Section~4.1]{BonLPP16a}, which applies a primal-dual approach to
\eqref{eq:quadratic} to satisfy it. In this connection, note that if
we have access to a lower bound $Q_{LB} \le Q^*$ (obtained by finding
a feasible point for the dual of \eqref{eq:quadratic}, or other
means), then any $d$ satisfying $Q(d) \le (1-\eta) Q_{LB}$ also
satisfies \eqref{eq:approx}.

In practical situations, we need not enforce \eqref{eq:approx}
explicitly for some chosen value of $\eta$. \modify{In fact, we do not necessarily require $\eta$ to be known, or \eqref{eq:approx} to be checked at all.  Rather, we can} take
advantage of the convergence rates of whatever solver is applied to
\eqref{eq:quadratic} to ensure that \eqref{eq:approx} holds for {\em
  some} value of $\eta \in (0,1)$, possibly unknown.  For instance, if
we apply an iterative solver to the strongly convex function $Q$ in
\eqref{eq:quadratic} that converges at a global linear rate
$(1-\tau)$, then the ``inner'' iteration sequence
$\{d^{(t)}\}_{t=0,1,\dotsc}$ (starting from some $d^{(0)}$ with
$Q(d^{(0)}) \le 0$) satisfies
\begin{equation}
	Q(d^{\left(t\right)}) - Q^*
\leq \left(1 - \tau\right)^t \left(Q\left(0\right) - Q^*\right), \quad
t=0,1,2,\dotsc.
\label{eq:titers}
\end{equation}
If we fix the number of inner iterations at $T$ (say), then $d^{(T)}$
satisfies \eqref{eq:approx} with $\eta= (1-\tau)^T$.  Although $\tau$
might be unknown as well, we can implicitly tune the accuracy of the
solution by adjusting $T$.  On the other hand, if we wish to attain a
certain target accuracy $\eta$ and have an estimate of rate $\tau$, we
can choose the number of iterations $T$ large enough that $(1-\tau)^T
\le \eta$. Note that $\tau$ depends on the extreme eigenvalues of
$H_k$ in some algorithms; we can therefore choose $H_k$ to ensure that
$\tau$ is restricted to a certain range for all $k$.

Empirically, we observe that Q-linear methods for solving
\eqref{eq:quadratic} often have rapid convergence in their early
stages, with slower convergence later.
Thus, a moderate value of $\eta$ may be preferable to a smaller value,
because moderate accuracy is attainable in disproportionately fewer
iterations than high accuracy.


A practical stopping condition for the subproblem solver in our
framework is just to set a fixed number of iterations, provided that a
linearly convergent method is used to solve \eqref{eq:quadratic}.
This guideline can be combined with other more sophisticated
approaches, possibly adjusting the number of inner iterations (and
hence implicitly $\eta$) according to some heuristics.  For
simplicity, our analysis assumes a fixed choice of $\eta \in
(0,1)$. We examine in particular the number of outer iterations
required to solve \eqref{eq:f} to a given accuracy $\epsilon$.  We
show that the dependence of the iteration complexity on the
inexactness measure $\eta$ is benign, increasing only modestly with
$\eta$ over approaches that require exact solution of
\eqref{eq:quadratic} for each $k$.

\subsection{Quadratic Approximation Algorithms} \label{sec:algs}

To build complete algorithms around the subproblem
\eqref{eq:quadratic}, we either do a step size line search along the
inexact solution $d^k$, or adjust $H_k$ and recompute $d^k$, seeking
in both cases to satisfy a familiar ``sufficient decrease'' criterion.
We present two algorithms that reflect each of these approaches.
The first uses a line search approach on the step size with a modified
Armijo rule, as presented in \cite{TseY09a}.  We consider a
backtracking line-search procedure for simplicity; the analysis could
be adapted for more sophisticated procedures.  Given the current point
$x^k$, the update direction $d^k$ and parameters $\beta, \gamma \in
(0,1)$, backtracking finds the smallest nonnegative integer $i$ such
that the step size $\alpha_k = \beta^i$ satisfies
\begin{equation}
F\left(x^k + \alpha_k d^k \right) \leq F\left( x^k \right) + \alpha_k \gamma
\Delta_k,
\label{eq:armijo}
\end{equation}
where
\begin{equation}
\Delta_k \coloneqq \nabla f\left( x^k \right)^Td^k + \psi\left( x^k + d^k
\right) - \psi\left( x^k \right).
\label{eq:Delta}
\end{equation}
This version appears as Algorithm~\ref{alg:QuasiArmijo}.  The exact
version of this algorithm can be considered as a special case of the
block-coordinate descent algorithm of \citet{TseY09a}.\footnote{The
  definition of $\Delta$ in \cite{TseY09a} contains another term
  $\omega d^T H d/2$, where $\omega \in [0,1)$ is a parameter.  We
    take $\omega =0$ for simplicity, but our analysis can be extended
    in a straightforward way to the case of $\omega \in (0,1)$.}  In
  \cite{BonLPP16a}, Algorithm~\ref{alg:QuasiArmijo} (with possibly a
  different criterion on $d^k$) is called the ``variable metric
  inexact line-search-based method''.  (We avoid the term ``metric''
  because we consider the possibility of indefinite $H_k$ in some of
  our results.) More complicated metrics, not representable by a
  matrix norm, were also considered in \cite{BonLPP16a}. Since our
  analysis makes use only of the smallest and largest eigenvalues of
  $H_k$ (which correspond to the strong convexity and Lipschitz
  continuity parameters of the quadratic approximation term), we could
  also generalize our approach to this setting.  We present only the
  matrix-representable case, however, as it allows a more direct
  comparison with the second algorithm presented next.

\begin{algorithm}
\begin{algorithmic}
\State Given $\beta, \gamma\in (0,1)$, $x^0 \in
	\R^n$;
\For{$k=0,1,2,\dotsc$}
	\State Choose a symmetric $H_k$ that makes $Q_k$ strongly convex;
	\State Obtain from \eqref{eq:quadratic}  a vector $d^k$ satisfying \eqref{eq:approx}, \modify{for some fixed $\eta \in [0,1)$};
	\State Compute $\Delta_k$ by \eqref{eq:Delta};
	\State $\alpha_k \leftarrow 1$;
	\While{\eqref{eq:armijo} is not satisfied}
		\State $\alpha_k \leftarrow \beta \alpha_k$;
	\EndWhile
	\State $x^{k+1} \leftarrow x^k + \alpha_k d^k$;
\EndFor
\end{algorithmic}
\caption{Inexact Successive Quadratic Approximation with Backtracking Line Search}
\label{alg:QuasiArmijo}
\end{algorithm}

The second algorithm uses the following sufficient decrease
criterion from \cite{SchT16a,GhaS16a}:
\begin{equation}
F\left(x\right) - F\left(x + d\right) \geq -\gamma
	Q^x_H\left(d\right) \geq 0,
\label{eq:decrease}
\end{equation}
for a given parameter $\gamma \in (0,1]$. If this criterion is not
satisfied, the algorithm modifies $H$ and recomputes $d^k$.  The
criterion \eqref{eq:decrease} is identical to that used by
trust-region methods (see, for example, \cite[Chapter 4]{NocW06a}), in
that the ratio between the actual objective decrease and the decrease
predicted by $Q$ is bounded below by $\gamma$; that is,
\begin{equation*}
\frac{F\left( x \right) - F\left( x + d \right)}{Q^x_H\left( 0
\right) - Q^x_H\left( d \right)} \geq \gamma.
\end{equation*}

We consider two variants of modifying $H$ such that
\eqref{eq:decrease} is satisfied. The first successively
increases $H$ by a factor $\beta^{-1}$ (for some parameter $\beta \in
(0,1)$) until \eqref{eq:decrease} holds.  We require in this variant
that the initial choice of $H$ is positive definite, so that all
eigenvalues grow by a factor of $\beta^{-1}$ at each multiplication.
The second variant uses a similar strategy, except that $H$ is
modified by adding a successively larger multiple of the identity,
until \eqref{eq:decrease} holds. (This algorithm allows negative
eigenvalues in the initial estimate of $H$.)  These two approaches are
defined as the first and the second variants of
Algorithm~\ref{alg:QuasiModifyH}, respectively.

\begin{algorithm}
\begin{algorithmic}[1]
\State Given $\beta, \gamma\in (0,1]$,  $x^0 \in
	\R^n$;
\For{$k=0,1,2,\dotsc$}
\If {Variant 1} Choose $H^0_k \succ 0$; \EndIf
\If {Variant 2} Choose a suitable $H^0_k$; \EndIf
\State $\alpha_k \leftarrow 1$, $H_k \leftarrow H^0_k$;
\State Obtain from \eqref{eq:quadratic}  a vector $d^k$ satisfying
\eqref{eq:approx}, \modify{for some fixed $\eta \in [0,1)$};
\While{\eqref{eq:decrease} is not satisfied}
	\If{Variant 1} $\alpha_k \leftarrow \beta \alpha_k$, $H_k
	\leftarrow  H^0_k/\alpha_k$; \EndIf
	\If{Variant 2} $H_k \leftarrow H^0_k + \alpha_k^{-1} I$, $\alpha_k
	\leftarrow \beta \alpha_k$; \EndIf
	\State Obtain from \eqref{eq:quadratic} a vector $d^k$ satisfying
		\eqref{eq:approx};
\EndWhile
\State $x^{k+1} \leftarrow x^k + d^k$;
\EndFor
\end{algorithmic}
\caption{Inexact Successive Quadratic Approximation with Modification
of the Quadratic Term}
\label{alg:QuasiModifyH}
\end{algorithm}


Algorithm~\ref{alg:QuasiArmijo} and Variant 1 of Algorithm
\ref{alg:QuasiModifyH} are direct extensions of backtracking line
search in the smooth case, in the sense that when $\psi$ is not present,
both approaches are identical to shrinking the step size.  However,
aside from the sufficient decrease criteria, the two differ when the
regularization term is present.

The second variant of Algorithm~\ref{alg:QuasiModifyH} is similar to
the method proposed in \cite{SchT16a,GhaS16a}, with the only
difference being the inexactness criterion of the subproblem solution.
This variant of modifying $H$ can be seen as interpolating between the
step from the original $H$ and the proximal gradient step.  It is also
a generalization of the trust-region technique for smooth
optimization.  When $\psi$ is not present, adding a multiple of the
identity to $H$ in \eqref{eq:quadratic} is equivalent to shrinking the
trust region \citep{MorS83a}. We can therefore think of
Algorithm~\ref{alg:QuasiModifyH}, Variant 2 as a generalized
trust-region approach for regularized problems.

Rather than our multiplicative criterion \eqref{eq:approx}, the works
\cite{SchT16a,GhaS16a} use an {\em additive} criterion to measure
inexactness of the solution.  In the analysis of
\cite{SchT16a,GhaS16a}, this tolerance must then be reduced to zero at
a certain rate as the algorithm progresses, resulting in growth of the
number of inner iterations per outer iteration as the algorithms
progress.  By contrast, we attain satisfactory performance (both in
theory and practice) for a fixed value $\eta \in (0,1)$ in
\eqref{eq:approx}.




Which of the algorithms described above is ``best'' depends on the
circumstances. When \eqref{eq:quadratic} is expensive to solve,
Algorithm~\ref{alg:QuasiArmijo} may be preferred, as it requires
approximate solution of this subproblem just once on each outer
iteration. On the other hand, when $\psi$ has special properties, such as
inducing sparsity or low rank in $x$, Algorithm \ref{alg:QuasiModifyH}
might benefit from working with sparse iterates and solving the
subproblem in spaces of reduced dimension.

Variants and special cases of the algorithms above have been discussed
extensively in the literature. Proximal gradient algorithms have $H =
\xi I$ for some $\xi > 0$ \citep{ComW05a,WriNF08a}; proximal-Newton uses
$H = \nabla^2 f$ \citep{LeeSS14a,RodK16a,LiAV17a};
proximal-quasi-Newton and variable metric use quasi-Newton
approximations for $H_k$ \citep{SchT16a,GhaS16a}.  The term
``successive quadratic approximation'' is also used by
\citet{ByrNO16a}.  Our methods can even be viewed as a special case of
block-coordinate descent \citep{TseY09a} with a single block.
The key difference in this work is the use of the inexactness criterion
\eqref{eq:approx}, while existing works either assume exact solution
of \eqref{eq:quadratic}, or use a different criterion that requires
increasing accuracy as the number of outer iterations grows. 
Some of these works provide only an asymptotic convergence guarantee
and a local convergence rate, with a lack of clarity about when the
fast local convergence rate will take effect.  An exception is
\cite{BonLPP16a}, which also makes use of the condition
\eqref{eq:approx}.  However, \cite{BonLPP16a} gives convergence rate
only for convex $f$ and requires existence of a scalar $\mu \ge 1$ and
a sequence $\{\zeta_k\}$ such that
\begin{equation}
\sum_{k=0}^\infty \zeta_k < \infty,\quad \zeta_k \geq 0,\quad H_{k+1}
\preceq \left(1 + \zeta_k\right) H_k, \quad \mu I \succeq H_k \succeq
\frac{1}{\mu}I, \quad \forall k,
\label{eq:assumption}
\end{equation}
where $A \succeq B$ means that $A-B$ is positive semidefinite.
This condition may preclude such useful and practical choices of $H_k$
as the Hessian and quasi-Newton approximations.  We believe that our
setting may be more general, practical, and straightforward in some
situations.

\subsection{Contribution}

This paper shows that, when the initial value of $H_k$ at all outer
iterations $k$ is chosen appropriately, and that \eqref{eq:approx} is
satisfied for all iterations, then the objectives of the two
algorithms converge at a global Q-linear rate under an ``optimal set
strong convexity'' condition defined in \eqref{eq:strong}, and at a
sublinear rate for general convex functions.  When $F$ is nonconvex,
we show sublinear convergence of the first-order optimality condition.
Moreover, to discuss the relation between the subproblem solution
precision and the convergence rate, we show that the iteration
complexity is proportional to either $1/(1 - \eta)$ or  $1/(2(1 -
\sqrt{\eta}))$, depending on the properties of $f$ and $\psi$, and
the algorithm parameter choices.\footnote{Note that for $\eta \in [0,1)$, $1
    / (1- \eta) > 1/ (2(1 - \sqrt{\eta}))$.}

In comparison to existing works, our major contributions are as follows.
\begin{itemize}
\item We quantify how the inexactness criterion \eqref{eq:approx}
  affects the step size of Algorithm~\ref{alg:QuasiArmijo}, the norm
  of the final $H$ in Algorithm~\ref{alg:QuasiModifyH}, and the
  iteration complexity of these algorithms.  We discuss why the
  process for finding a suitable value of $\alpha_k$ in each algorithm
  can potentially improve the convergence speed when the
  quadratic approximations incorporate curvature information,
  leading to acceptance of step sizes whose values are close to one.
\item We provide a global convergence rate result on the first-order
  optimality condition for the case of nonconvex $f$ in \eqref{eq:f}
  for general choices of $H_k$, without assumptions beyond the
  Lipschitzness of $\nabla f$.
\item The global R-linear convergence case of a similar algorithm in
  \cite{GhaS16a} when $F$ is strongly convex is improved to a global
  Q-linear convergence result for a broader class of problems.
\item For general convex problems, in addition to the known sublinear
  ($1/k$) convergence rate, we show linear convergence with a rate
  independent of the conditioning of the problem in the early stages
  of the algorithm.
 \item Faster linear convergence in the early iterations also applies
   to problems with global Q-linear convergence, explaining in part
   the empirical observation that many methods converge rapidly in
   their early stages before settling down to a slower rate.  This
   observation also allows improvement of iteration
   complexities. 
\end{itemize}

\subsection{Related Work}
\label{sec:related}

Our general framework and approach, and special cases thereof, have
been widely studied in the literature. Some related work has already
been discussed above. We give a broader discussion in this section.

When $\psi$ is the indicator function of a convex constraint set, our
approach includes an inexact variant of a constrained Newton or quasi-Newton
method. There are a number of papers on this approach, but their
convergence results generally have a different flavor from ours.  They
typically show only asymptotic convergence rates, together with global
convergence results without rates, under weaker smoothness and
convexity assumptions on $f$ than we make here.  For example, when $\psi$
is the indicator function of a ``box'' defined by bound constraints,
\citet{ConGT88a} applies a trust-region framework to solve
\eqref{eq:quadratic} approximately, and shows asymptotic convergence.
The paper \cite{ByrLNZ95a} uses a line-search approach, with $H_k$
defined by an L-BFGS update, and omits convergence results.  For
constraint sets defined by linear inequalities, or general convex
constraints, \cite{BurMT90a} shows global convergence of a trust
region method using the Cauchy point. A similar approach using the
exact Hessian as $H_k$ is considered in \cite{LinM99a}, proving local
superlinear or quadratic convergence in the case of linear
constraints.

Turning to our formulation \eqref{eq:f} in its full generality,
Algorithm~\ref{alg:QuasiArmijo} is analyzed in \cite{BonLPP16a}, which
refers to the condition \eqref{eq:approx} as ``$\eta$-approximation.''
(Their $\eta$ is equivalent to $1 - \eta$ in our notation.) This paper
shows asymptotic convergence of $Q_k(d)$ to zero without requiring
convexity of $F$, Lipschitz continuity of $\nabla f$, or a fixed value
of $\eta$.  The only assumptions are that $Q_k(d^k) < 0$ for all $k$
and the sequence of objective function values converges (which always
happens when $F$ is bounded below).
Under the additional assumptions that $\nabla f$ is Lipschitz
continuous, $F$ is convex,
\eqref{eq:assumption}, and \eqref{eq:approx}, they showed convergence
of the objective value at a $1/k$ rate. The same authors considered
convergence for nonconvex functions satisfying a
Kurdyka-{\L}ojasiewicz condition in \cite{BonLPPR17a}, but the exact
rates are not given. Our results differ in not requiring the
assumption \eqref{eq:assumption}, and we are more explicit about the
dependence of the rates on $\eta$.  Moreover, we show detailed
convergence rates for several additional classes of problems.

A version of Algorithm~\ref{alg:QuasiModifyH} without line search but
requiring $H_k$ to {\em overestimate} the Hessian, as follows:
\begin{equation*}
	f(x^k + d) \leq f(x^k) + \nabla f(x^k)^T d + \frac12 d^T H_k d
\end{equation*}
is considered in \cite{ChoPR14a}.
  Asymptotic convergence is proved, but no rates are given.

Convergence of an inexact proximal-gradient method (for which $H_k = L
I$ for all $k$) is discussed in \cite{SchRB11a}.  With this choice of
$H_k$, \eqref{eq:decrease} always holds with $\gamma = 1$.  They also
discuss its accelerated version for convex and strongly convex
problems.  Instead of our multiplicative inexactness criterion, they
assume an additive inexactness criterion in the subproblem, of the
form
\begin{equation}
	Q_k\left(d^k\right) \leq Q_k^* + \epsilon_k.
	\label{eq:additive}
\end{equation}
Their analysis also allows for an error $e^k$ in the gradient term
in \eqref{eq:quadratic}.  The paper shows that for general convex
problems, the objective value converges at a $1/k$ rate provided that
$\sum_k \sqrt{\epsilon_k}$ and $\sum_k \|e^k\|$ converge.  For
strongly convex problems, they proved R-linear convergence of $\|x^k -
x^*\|$, provided that the sequence $\{ \|e^k\| \}$ and $\{
\sqrt{\epsilon_k}\}$ both decrease linearly to zero.
When our approaches are specialized to proximal gradient ($H_k = LI$),
our analysis shows a Q-linear rate (rather than R-linear) for the
strongly convex case, and applies to the convergence of the objective
value rather than the iterates. Additionally, our results shows
convergence for nonconvex problems.


Variant 2 of Algorithm~\ref{alg:QuasiModifyH} is proposed in
\cite{SchT16a,GhaS16a} for convex and strongly convex objectives, with
inexactness defined additively as in \eqref{eq:additive}. For convex
$f$, \citet{SchT16a} showed that if $\sum_{k=0}^\infty
\epsilon_k/\|H_k\|$ and $\sum_{k=0}^{\infty}
\sqrt{\epsilon_k/\|H_k\|}$ converge then a $1/k$ convergence rate is
achievable. The same rate can be achieved if $\epsilon_k \leq (a/k)^2$
for any $a \in [0,1]$.  When $F$ is $\mu$-strongly convex,
\citet{GhaS16a} showed that if $\sum \epsilon_k / \rho^k$ is finite
(where $\rho = 1 - (\gamma \mu)/(\mu + M)$, $M$ is the upper bound for
$\|H_k\|$, and $\gamma$ is as defined in \eqref{eq:decrease}), then a
global R-linear convergence rate is attained.  In both cases, the
conditions require a certain rate of decrease for $\epsilon_k$, a
condition that can be achieved by performing more and more inner
iterations as $k$ increases.  By contrast, our multiplicative
inexactness criterion \eqref{eq:approx} can be attained with a fixed
number of inner iterations. Moreover, we attain a Q-linear rather than
an R-linear result.


Algorithm~\ref{alg:QuasiArmijo} is also considered in \cite{LeeSS14a},
with $H_k$ set either to $\nabla^2 f(x^k)$ or a BFGS
approximation. Asymptotic convergence and a local rate are shown for
the exact case. For inexact subproblem solutions, local results are
proved under the assumption that the unit step size is always taken
(which may not happen for inexact steps).  A variant of
Algorithm~\ref{alg:QuasiArmijo} with a different step size criterion
is discussed in \cite{ByrNO16a}, for the special case of $\psi(x) =
\|x\|_1$. Inexactness of the subproblem solution is measured by the
norm of a proximal-gradient step for $Q$. By utilizing specific
properties of the $\ell_1$ norm, this paper showed a global
convergence rate on the norm of the proximal gradient step on $F$ to
zero, without requiring convexity of $f$ --- a result similar to our
nonconvex result.  However, the extension of their result to general
$\psi$ is not obvious and, moreover, our inexactness condition avoids
the cost of computing the proximal gradient step on $Q$.  When $H_k$
is $\nabla^2 f(x^k)$ or a BFGS approximation, they obtain for the
inexact version local convergence results similar to the exact case
proved in \cite{LeeSS14a}.

For the case in which $f$ is convex, thrice continuously
differentiable, and self-concordant, and $\psi$ is the indicator
function of a closed convex set, \citet{TraKC14a} analyzed global and
local convergence rates of inexact damped proximal Newton with a fixed
step size.  The paper \citet{LiAV17a} extends this convergence
analysis to general convex $\psi$.
However, generalization of these results beyond the case of $H_k =
\nabla^2 f(x^k)$ and self-concordant $f$ is not obvious.

Accelerated inexact proximal gradient is discussed in
\cite{SchRB11a,VilSBV13a} for convex $f$ to obtain an improved
$O(1/k^2)$ convergence rate.  The work \cite{JiaST12a} considers
acceleration with more general choices of $H$ under the requirement
$H_k \succeq H_{k+1}$ for all $k$, which precludes many interesting
choices of $H_k$.  This requirement is relaxed by \cite{GhaS16a} to
$\theta_k H_k \succeq \theta_{k+1} H_{k+1}$ for scalars $\theta_k$
that are used to decide the extrapolation step size.  However, as
shown in the experiment in \cite{GhaS16a}, extrapolation may not
accelerate the algorithm.  Our analysis does not include acceleration
using extrapolation steps, but by combining with the Catalyst
framework \citep{LinMH15a}, similar improved rates could be attained.

\subsection{Outline: Remainder of the Paper}

In Section~\ref{sec:preliminaries}, we introduce notation and prove
some preliminary results.  Convergence analysis appears in
Section~\ref{sec:analysis} for Algorithms~\ref{alg:QuasiArmijo} and
\ref{alg:QuasiModifyH}, covering both convex and nonconvex problems.
Some interesting and practical choices of $H_k$ are discussed in
Section~\ref{sec:H} to show that our framework includes many existing
algorithms.  We provide some preliminary numerical results in
Section~\ref{sec:exp}, and make some final comments in
Section~\ref{sec:conclusions}.

\section{Notations and Preliminaries}
\label{sec:preliminaries}

The norm $\|\cdot\|$, when applied on vectors, denotes the Euclidean
norm. When applied to a symmetric matrix $A$, it denotes the
corresponding induced norm, which is equivalent to the spectral radius
of $A$.  For any symmetric matrix $A$, $\lambdamin(A)$ denotes its
smallest eigenvalue.  For any two symmetric matrices $A$ and $B$,
$A\succeq B$ (respectively $A \succ B$) denotes that $A-B$ is positive
semidefinite (respectively positive definite).  For our nonsmooth
function $F$, $\partial F$ denotes the set of generalized gradient
defined as
\begin{equation*}
	\partial F(x) \coloneqq \nabla f(x) + \partial \Psi(x),
\end{equation*}
where $\partial \Psi$ denotes the subdifferential (as $\Psi$ is convex).
When the minimum $F^*$ of $F(x)$ is attainable, we denote the
solution set by $\Omega \coloneqq \left\{x\mid F\left( x \right) =
F^*\right\}$, and define $P_{\Omega}(x)$ as the (Euclidean-norm)
projection of $x$ onto $\Omega$.

In some results, we use a particular strong convexity assumption to
obtain a faster rate. We say that $F$ satisfies the {\em optimal set
  strong convexity} condition with modulus $\mu \ge 0$ if for any $x$
and any $\lambda \in [0,1]$, we have
\begin{align}
F(\lambda x +  (1 - \lambda)
P_{\Omega}(x))
\leq \lambda F\left(x\right) +
\left(1 - \lambda\right) F^* - \frac{\mu\lambda \left( 1 - \lambda
\right)}{2} \left\|x -
P_{\Omega}\left(x\right)\right\|^2.
\label{eq:strong}
\end{align}
This condition does not require the strong convexity to hold globally,
but only between the current point and its projection onto the
solution set.
Examples of functions that are not strongly convex but satisfy
\eqref{eq:strong} include:
\begin{itemize}
\item $F(x) = h(Ax)$ where $h$ is strongly convex, and $A$ is any
	matrix;
\item $F(x) = h(Ax) + \bfone_{X}(x)$, where $X$ is a polyhedron;
\item Squared-hinge loss: $F(x) = \sum \max (0, a_i^T x -b_i)^2$.
\end{itemize}
A similar condition is the ``quasi-strong convexity'' condition
proposed by \citet{NecNG18a}, which always implies \eqref{eq:strong},
and can be implied by optimal set strong convexity if $F$ is
differentiable. However, since we allow $\psi$ (and therefore $F$) to
be nonsmooth, we need a different definition here.

Turning to the subproblem \eqref{eq:quadratic} and the definition of
$\Delta_k$ in \eqref{eq:Delta}, we find a condition for $d$ to be a
descent direction.
\begin{lemma}
\label{lemma:delta}
\modify{
If $\Psi$ is convex and $f$ is differentiable, then $d$ is a descent
direction for $F$ at $x$ if $\Delta < 0$.}
\end{lemma}
\begin{proof}
\modify{
We know that $d$ is a descent direction for $F$ at $x$ if the
directional derivative
\begin{equation*}
	F'(x;d) \coloneqq \lim_{\alpha \rightarrow 0}\quad \frac{F(x +
		\alpha d) -
	F(x)}{\alpha}
\end{equation*}
is negative.
Note that since $f$ is differentiable and $\Psi$ is convex,
\[
F'(x;d) = \max_{s \in \partial F(x)} s^T d
= \nabla f(x)^T d + \max_{\hat s \in \partial \Psi(x)} \hat{s}^T d
\]
is well-defined.
Now from the convexity of $\Psi$,
\begin{equation*}
	\Psi(x+d) \geq \Psi(x) + \hat s^T d, \quad \forall \hat s \in \partial
	\Psi(x),
\end{equation*}
so
\begin{equation*}
	\max_{\hat s \in \partial \Psi(x)} \hat s^T d + \nabla f(x)^T d
	\leq \Psi(x+d) - \Psi(x) + \nabla f(x)^T d = \Delta.
\end{equation*}
Therefore, when $\Delta < 0$, the directional derivative is negative
and $d$ is a descent direction.
\qed
}
\end{proof}
The following lemma motivates our algorithms.
\begin{lemma}
\label{lemma:descent}
\modify{If $Q$ and $\Psi$ are convex and $f$ is differentiable}, then
$Q(d) < 0$ implies that $d$ is a descent direction for $F$ at $x$.
\end{lemma}
\begin{proof}
Note that $Q(0) = 0$.  Therefore, if $Q$ is convex, we have
\begin{equation*}
\lambda \nabla f\left(x\right)^T d + \frac{\lambda^2}{2} d^T H d +
\psi\left(x + \lambda d\right) - \psi\left(x\right) = Q^x_H\left(\lambda
d\right) \leq \lambda Q^x_H\left(d\right) < 0,
\end{equation*}
for all $\lambda \in (0,1]$.  It follows that $\nabla f(x)^T (\lambda
d) + \psi(x + \lambda d) - \psi(x) < 0$ for all sufficiently small
$\lambda$.  Therefore, from Lemma~\ref{lemma:delta}, $\lambda d$ is a
descent direction, and since $d$ and $\lambda d$ only differ in their
lengths, so is $d$.
	\qed
\end{proof}

Positive semidefiniteness of $H$ suffices to ensure convexity of $Q$.
However, Lemma~\ref{lemma:descent} may be used even when $H$ has
negative eigenvalues, as $\psi$ may have a strong convexity property
that ensures convexity of $Q$.  Lemma~\ref{lemma:descent} then
suggests that no matter how coarse the approximate solution of
\eqref{eq:quadratic} is, as long as it is better than $d=0$ for a
convex $Q$, it results in a descent direction. This fact implies
finite termination of the backtracking line search procedure in
Algorithm~\ref{alg:QuasiArmijo}.

\section{Convergence Analysis}
\label{sec:analysis}

We start our analysis for both algorithms by showing finite
termination of the line search procedures.  We then discuss separately
three classes of problems involving different assumptions on $F$,
namely, that $F$ is convex, that $F$ satisfies optimal set strong
convexity \eqref{eq:strong}, and that $F$ is nonconvex. Different
iteration complexities are proved in each case.
\modify{
The following condition is assumed throughout our analysis in this section.
\begin{assumption}
	\label{assum:general}
	In \eqref{eq:f}, $f$ is $L$-Lipschitz-continuously differentiable
	for some $L > 0$;
$\psi$ is convex, extended-valued, proper, and closed;
$F$ is lower-bounded; and the solution set $\Omega$ of \eqref{eq:f} is
nonempty.
\end{assumption}}

\subsection{Line Search Iteration Bound}
\label{subsec:linesearch}

We show that the line search procedures have finite termination.  The
following lemma for the backtracking line search in
Algorithm~\ref{alg:QuasiArmijo} does not require $H$ to be positive
definite, though it does require strong convexity of $Q$
\eqref{eq:quadratic}.
\begin{lemma}
\label{lemma:delta_d}
\modify{If Assumption \ref{assum:general} holds,} $Q$ is $\sigma$-strongly
convex for some $\sigma > 0$, and the approximate solution $d$ to
\eqref{eq:quadratic} satisfies \eqref{eq:approx} for some
$\eta < 1$, then for $\Delta$ defined in \eqref{eq:Delta}, we have
\begin{align}
	\nonumber
\Delta &\leq -\frac12 \left(\frac{1 - \sqrt{\eta}}{
	1 + \sqrt{\eta}} \sigma \left\|d\right\|^2 + d^T H d\right)\\
&\leq -\frac12 \left(\frac{1 - \sqrt{\eta}}{
	1 + \sqrt{\eta}} \sigma + \lambdamin\left( H \right) \right)
	\left\|d\right\|^2.
\label{eq:deltabound}
\end{align}
Moreover, if
\begin{equation*}
(1 - \sqrt{\eta})\sigma + (1 + \sqrt{\eta}) \lambdamin(H) > 0,
\end{equation*}
then the backtracking line search procedure in
Algorithm~\ref{alg:QuasiArmijo} terminates in finite steps and
produces a step size $\alpha$ that satisfies the following lower
bound:
\begin{equation}
\alpha \geq \min\left\{1, \beta\left(1 - \gamma\right) \frac{
	\left(1 - \sqrt{\eta}\right)\sigma + \left( 1 + \sqrt{\eta}
	\right)\lambdamin\left( H \right)}{L\left(1 + \sqrt{
	\eta } \right) } \right\}.
\label{eq:linesearchbound}
\end{equation}
\end{lemma}
\begin{proof}
From \eqref{eq:approx} and strong convexity of $Q$, we have
that for any $\lambda \in [0,1]$,
\begin{align}
\frac{1}{1 - \eta}\left(Q\left(0\right) - Q\left(d\right)\right) &
\geq Q\left(0\right) - Q^*\nonumber\\
& \geq Q\left(0\right) - Q\left(\lambda d\right)
\label{eq:ls1}
\\
& \geq Q\left(0\right) - \left(\lambda Q\left(d\right) + \left(1 -
\lambda\right) Q\left(0\right) - \frac{\sigma \lambda \left( 1 -
\lambda \right)}{2} \left\| d \right\|^2 \right).
\nonumber
\end{align}
Since $Q(0) = 0$, we obtain by substituting from the definition of $Q$
that
\begin{align*}
&~\frac{1}{1 - \eta} \left( \nabla f\left(x\right)^T d + \frac{1}{2} d^T H d +
	\psi\left(x + d\right) - \psi\left(x\right) \right)
\\
\leq &~\lambda \left(\nabla f\left(x\right)^T d + \frac{1}{2}d^T H d +
	\psi\left(x+d\right) - \psi\left(x\right)\right) - \frac{\sigma \lambda
	\left(1 - \lambda\right)}{2} \left\|d\right\|^2.
\end{align*}
Since $1/(1-\eta) \geq 1 \geq \lambda$, we have
\begin{align}
	\nonumber
\left(\frac{1}{1 - \eta} - \lambda\right) \Delta
&\leq -\frac{\sigma \lambda \left(1 - \lambda \right) }{2} \left\| d
	\right\|^2 + \frac12 \left( \lambda - \frac{1}{1 - \eta} \right) d^T H d\\
&\leq -\left(\frac{\sigma \lambda \left(1 - \lambda \right) }{2}
	+ \frac12 \left(\frac{1}{1 - \eta} - \lambda \right) \lambdamin\left( H
	\right)\right) \left\| d \right\|^2.
\label{eq:deltaintermediate}
\end{align}
It follows immediately that the following bound holds for any $\lambda
\in [0,1]$:
\begin{equation*}
\Delta
\leq -\frac12 \left(\frac{\sigma \lambda \left(1 - \lambda \right) }{
\left(\frac{1}{1-\eta} - \lambda \right)} + \lambdamin \left( H \right)
\right) \left\| d \right\|^2.
\end{equation*}
We make the following specific choice of $\lambda$:
\begin{equation}
\lambda = \frac{1-\sqrt{\eta}}{1-\eta} \in (0,1].
\label{eq:lambda}
\end{equation}
for which
\[
1-\lambda = \sqrt{\eta} \lambda, \quad
\frac{1}{1-\eta} - \lambda = \frac{\sqrt{\eta}}{1-\eta}.
\]
The result \eqref{eq:deltabound} follows by substituting these
identities into \eqref{eq:deltaintermediate}.


\modify{ If the right-hand side of \eqref{eq:deltabound} is negative,
  then we have from the Lipschitz continuity of $\nabla f$, the
  convexity of $\psi$, and the mean value theorem that the following
  relationships are true for all $\alpha \in [0,1]$:
\begin{align*}
\nonumber
&~F\left(x + \alpha d\right) - F\left(x\right)\\
=&~ f\left(x + \alpha d \right) - f\left(x \right) + \psi\left(x +
\alpha d \right) - \psi\left(x\right) \nonumber\\
\leq &~ \alpha \nabla f\left(x\right)^T d - \alpha \left(\psi\left(x\right)
	- \psi\left(x+d\right)\right) + \alpha \int_0^1  \left( \nabla
	f\left(x+t\alpha d\right) - \nabla f\left(x\right)\right)^T d \, dt\nonumber\\
\leq&~ \alpha \Delta + \frac{L \alpha^2}{2} \left\|d\right\|^2\\
\leq&~ \alpha \Delta - \frac{L \alpha^2(1 + \sqrt{\eta})}{\left(1 -
	\sqrt{\eta}\right)\sigma  + \left( 1 + \sqrt{\eta}
	\right)\lambdamin\left( H \right)}\Delta.
\end{align*}
}
Therefore,
\eqref{eq:armijo} is satisfied if
\begin{equation*}
\alpha \Delta -
\frac{L \alpha^2(1 + \sqrt{\eta})}{ \left(1 - \sqrt{\eta} \right)
	\sigma + \left( 1 + \sqrt{\eta} \right)\lambdamin\left( H
	\right)}\Delta
\leq \alpha \gamma \Delta.
\end{equation*}
We thus get that \eqref{eq:armijo} holds whenever
\begin{equation*}
\alpha \leq (1-\gamma) \frac{\left(1 -
	\sqrt{\eta}\right) \sigma + \left( 1 + \sqrt{\eta}
	\right)\lambdamin\left( H \right)}{L\left(1 +
	\sqrt{\eta}\right)}.
\end{equation*}
This leads to \eqref{eq:linesearchbound}, when we introduce a factor
$\beta$ to account for possible undershoot of the backtracking
procedure.
	\qed
\end{proof}


Note that Lemma~\ref{lemma:delta_d} allows indefinite $H$, and
suggests that we can still obtain a certain amount of objective
decrease as long as $\lambdamin(H)$ is not too negative in comparison
to the strong convexity parameter of $Q$.  When the strong convexity
of $Q$ is accounted for completely by the quadratic part (that is,
$\lambdamin(H) = \sigma>0$) we have the following simplification of
Lemma~\ref{lemma:delta_d}.
\begin{corollary}
\label{lemma:delta_d2}
\modify{If Assumption \ref{assum:general} holds}, $\lambdamin(H) = \sigma > 0$, 
and the approximate solution $d$ to \eqref{eq:quadratic} satisfies
\eqref{eq:approx} for some $\eta < 1$, we have
\begin{equation}
\Delta \leq -\frac{1}{1 + \sqrt{\eta}} d^T H d \leq
	-\frac{\sigma}{1 + \sqrt{\eta}} \left\|d\right\|^2.
\label{eq:deltabound2}
\end{equation}
Moreover, the backtracking line search procedure in
Algorithm~\ref{alg:QuasiArmijo} terminates in finite steps and
produces a step size that satisfies the following lower bound:
\begin{equation}
\alpha \geq \bar \alpha \coloneqq \min\left\{1,
	\frac{2\beta\left(1 - \gamma\right) \sigma}{L\left(1 +
		\sqrt{\eta}\right)}\right\}.
\label{eq:linesearchbound2}
\end{equation}
\end{corollary}
\begin{proof}
Following \eqref{eq:ls1}, we have from convexity of $\psi$ for any
$\lambda \in [0,1]$ that
\begin{align*}
&~\frac{1}{1 - \eta} \left( \nabla f\left(x\right)^T d + \frac{1}{2}d^T H d +
	\psi\left(x + d\right) - \psi\left(x\right) \right)\\
\leq&~ \lambda \left(\nabla f\left(x\right)^T d + \frac{\lambda}{2}d^T H
	d + \psi\left(x+d\right) - \psi\left(x\right)\right).
\end{align*}
Therefore,
\begin{equation}
\left(\frac{1}{1 - \eta} - \lambda\right) \Delta
\leq \left(\lambda^2 - \frac{1}{1 - \eta}\right)\frac{1}{2} d^T H d.
\label{eq:delta_intermediate}
\end{equation}
Using \eqref{eq:lambda} in \eqref{eq:delta_intermediate}, we obtain
\eqref{eq:deltabound2}.  The bound \eqref{eq:linesearchbound2} follows
by substituting $\sigma = \lambdamin(H)$ into \eqref{eq:linesearchbound}.
\qed
\end{proof}

Note that the first inequality in \eqref{eq:deltabound} and the second
inequality in \eqref{eq:deltabound2} make use of the pessimistic lower
bound $d^THd \ge \lambdamin(H) \|d\|^2$, 
in practice, we observe (see Section~\ref{sec:exp}) that the unit step
$\alpha_k=1$ is often accepted in practice (significantly larger than
the lower bounds \eqref{eq:linesearchbound} and
\eqref{eq:linesearchbound2}) when $H_k$ is the actual Hessian
$\nabla^2 f(x^k)$ or its quasi-Newton approximation.


Next we consider Algorithm~\ref{alg:QuasiModifyH}.
\begin{lemma}
\label{lemma:Hbound}
\modify{If Assumption \ref{assum:general} holds}, $Q$ is $\sigma$-strongly convex for some $\sigma > 0$, and $d$ is
an approximate solution to \eqref{eq:quadratic} satisfying
\eqref{eq:approx} for some $\eta \in [0,1)$, then \eqref{eq:decrease}
  is satisfied if
\begin{equation}
\left( 1 - \gamma \right)  \frac{1 - \sqrt{\eta}}{1 +
	\sqrt{\eta}} \sigma  + \lambdamin(H)  \geq L.
\label{eq:Hbound}
\end{equation}
Therefore, in Algorithm~\ref{alg:QuasiModifyH}, if the initial $H^0_k$ satisfies
\begin{equation}
m_0 I \preceq H^0_k \preceq M_0 I
\label{eq:H0}
\end{equation}
for some $M_0 > 0$, $m_0 \leq M_0$, then for Variant 2, the final
$H_k$ satisfies
\begin{equation}
\|H_k\| \leq \tilde{M}_2(\eta) \coloneqq M_0 + \max\left\{1,
	\frac{1}{\beta}\left(\frac{L \left( 1 + \sqrt{\eta} \right)}{2 -
		\gamma \left( 1 - \sqrt{\eta} \right)} - m_0\right) \right\}.
\label{eq:m2}
\end{equation}
For Variant 1, if we assume in addition that $m_0>0$, we have
\begin{equation}
\|H_k\| \leq \tilde{M}_1(\eta) \coloneqq M_0 \max\left\{1, \frac{L \left( 1 +
	\sqrt{\eta}\right)}{\beta \left(2 - \gamma \left( 1 - \sqrt{\eta}
\right)\right)m_0}\right\}.
\label{eq:m1}
\end{equation}
\end{lemma}
\begin{proof}
	From Lipschitz continuity of $\nabla f$, we have that
\begin{align}
\nonumber
&~F\left(x\right) - F\left(x+d\right) + \gamma Q^x_H\left(d\right) \\
\nonumber
= &~f\left(x\right) - f\left(x+d\right) + \gamma \nabla
f\left(x\right)^T d + \frac{\gamma}{2} d^T H d +  \left(1 -
\gamma\right)\left(\psi\left(x\right) - \psi\left(x+d\right)\right)\\
\nonumber
\geq&~ \left(\gamma - 1\right) \nabla f\left(x\right)^T d -
	\frac{L}{2} \|d\|^2 + \frac{\gamma}{2}d^T H d +\left(1 -
	\gamma\right) \left(\psi\left(x\right) - \psi\left(x+d\right)\right)\\
\label{eq:Hbound1}
=&~ (\gamma - 1) \Delta - \frac{L}{2} \|d\|^2 + \frac{\gamma}{2} d^T
H d\\
\label{eq:Hbound2}
\geq&~ \frac{1 - \gamma}{2} \left( \frac{1 - \sqrt{\eta}}{1 + \sqrt{\eta}} \sigma \|d\|^2 + d^T H d  \right) -
\frac{L}{2} \|d\|^2
+ \frac{\gamma }{2} d^T H d,
\end{align}
where in \eqref{eq:Hbound1} we used the definition \eqref{eq:Delta},
and in \eqref{eq:Hbound2} we used Lemma~\ref{lemma:delta_d}.  By
noting $d^T H d \geq \lambdamin(H) \|d\|^2$, \eqref{eq:Hbound2} shows
that \eqref{eq:Hbound} implies \eqref{eq:decrease}.

Since $\psi$ is convex, we have that
$\sigma \ge \lambdamin(H)$, so that a sufficient condition for
\eqref{eq:Hbound} is that
\begin{equation*}
\left(\left( 1 - \gamma \right) \frac{1 -
\sqrt{\eta}}{1 + \sqrt{\eta}} + 1 \right) \lambdamin(H) \geq L,
\end{equation*}
which is equivalent to
\begin{equation*}
	\frac{2 - \gamma(1 - \sqrt{\eta})}
	{1 + \sqrt{\eta}} \lambdamin(H) \geq L.
\end{equation*}
Let the coefficient of $\lambdamin(H)$ in the above inequality be
denoted by $C_1$,
this observation suggests that for
Variant 1 the smallest eigenvalue of the final
$H$ is no larger than $L/(C_1 \beta)$, and since the proportion
between the largest and the smallest eigenvalues of $H_k$ remains
unchanged after scaling the whole matrix, we  obtain \eqref{eq:m1}.

For Variant 2, to satisfy $C_1 H \succeq L I$,
the coefficient for $I$ must be at least $L/C_1 - m_0$.  Considering
the overshoot, and that the difference between the largest and the
smallest eigenvalues is fixed after adding a multiple of identity, we
obtain the condition \eqref{eq:m2}.
	\qed
\end{proof}

By noting the simplification from $d^T H d \ge \lambdamin(H) \|d\|^2$,
we rarely observe the worst-case bounds \eqref{eq:m1} or \eqref{eq:m2}
in practice, unless $H^0$ is a multiple of the identity.

\subsection{Iteration Complexity}

Now we turn to the iteration complexity of our algorithms, considering
three different assumptions on $F$: convexity, optimal set strong
convexity, and the general (possibly nonconvex) case.


The following lemma is modified from some intermediate results in
\cite{GhaS16a}, which shows R-linear convergence of Variant 2 of
Algorithm~\ref{alg:QuasiModifyH} for a strongly convex objective when
the inexactness is measured by an additive criterion.
A proof can be found in Appendix~\ref{app:lemmaq}.
\begin{lemma}
\label{lemma:Q}
Let $F^*$ be the optimum of $F$. \modify{If Assumption \ref{assum:general} holds}, $f$ is convex and $F$ is
$\mu$-optimal-set-strongly convex as defined in \eqref{eq:strong} for
some $\mu \geq 0$, then for any given $x$ and $H$, and for all $\lambda \in [0,1]$,
we have
\begin{align}
\nonumber
Q^* \leq&~ \lambda \left(F^* - F\left(x\right)\right) - \frac{\mu
	\lambda \left( 1 - \lambda\right)}{2} \left\|x -
	P_{\Omega}\left(x\right)\right\|^2\\
	\nonumber
	&~\qquad\qquad+\frac{\lambda^2}{2} \left( x
	- P_{\Omega}\left(x\right) \right)^T H \left( x - P_{\Omega}\left(
	x \right) \right)\\
\leq &~ \lambda \left( F^* - F\left( x \right) \right) +
	\frac12 \left\| x - P_{\Omega}\left( x \right)\right\|^2
	\left( \left\|H\right\| \lambda^2 - \mu \lambda \left( 1 - \lambda
	\right) \right),\;
\label{eq:Q}
\end{align}
where $Q^*$ is the optimal objective value of \eqref{eq:quadratic}.
In particular, by setting $\lambda = \mu/(\mu+\|H\|)$ (as in
\cite{GhaS16a}), we have
\begin{equation} \label{eq:Qstrong}
Q^* \le \frac{\mu}{\mu+\|H\|} (F^*-F(x)).
\end{equation}
\end{lemma}
Note that we allow $\mu = 0$ in Lemma~\ref{lemma:Q}.


\subsubsection{Sublinear Convergence for General Convex Problems}
\label{sec:scgc}

We start with case of $F$ convex, that is, $\mu = 0$ in the definition
\eqref{eq:strong}. In this case, the first inequality in \eqref{eq:Q}
reduces to
\begin{equation}
Q_k^* \leq \lambda \left(F^* - F\left(x^k\right)\right) + \lambda^2
\frac{\left( x^k- P_{\Omega}\left( x^k \right) \right)^T H_k \left( x^k-
P_{\Omega}\left( x^k \right) \right)}{2},
\label{eq:Qlambda}
\end{equation}
for all $\lambda \in [0,1]$.  We assume the following in this
subsection.
\begin{assumption}
	\label{assum:Q2}
	There exists finite $R_0, M > 0$ such that
\begin{equation}
\sup_{x: F\left( x \right) \leq F\left( x_0
	\right)}\left\|x - P_\Omega(x)\right\| = R_0 < \infty \;\; \mbox{and} \;\;
	\|H_k\| \leq M, \;\; k=0,1,2,\dotsc.
	\label{eq:R0}
\end{equation}
\end{assumption}
Using this assumption, we can bound the second term in
\eqref{eq:Qlambda} by
\begin{equation}
\hat A \coloneqq \sup_k \left( x^k - P_{\Omega}\left( x^k \right) \right)^T
H_k \left( x^k - P_{\Omega}\left( x^k \right) \right) \leq M R_0^2.
\label{eq:A}
\end{equation}
The bound $\hat A \leq M R_0^2$ is quite pessimistic, but we use it
for purposes of comparing with existing works.

The following lemma is inspired by \cite[Lemma 4.4]{Bac15a} but
contains many nontrivial modifications, and will be needed in proving
the convergence rate for general convex problems.  Its proof can be
found in Appendix~\ref{app:lemmaseq}.
\begin{lemma}
\label{lemma:seq}
Assume we have three nonnegative sequences $\{\delta_k\}_{k \geq 0}$,
$\{c_k\}_{k \geq 0}$, and $\{A_k\}_{k \geq 0}$, and a constant $A > 0$
such that for all $k=0,1,2,\dotsc$, and for all $\lambda_k \in [0,1]$, we have
\begin{equation}
	 0 < A_k \leq A, \quad \delta_{k+1} \leq \delta_k +
         c_k\left(-\lambda_k \delta_k +
         \frac{A_k}{2}\lambda_k^2\right).
\label{eq:seq}
\end{equation}
Then for $\delta_k \geq A_k$, we have
\begin{equation}
\delta_{k+1} \leq \left(1 - \frac{c_k}{2}\right) \delta_{k}.
\label{eq:linear}
\end{equation}
In addition, if we define $k_0 \coloneqq \arg\min \{k: \delta_k <
A\}$, then
\begin{equation}
	\label{eq:sub}
	\delta_k \leq \frac{2 A}{\sum_{t=k_0}^{k-1} c_t+2}, \quad
	\mbox{\rm for all $k \ge k_0$.}
\end{equation}
\end{lemma}

By Lemma~\ref{lemma:seq} together with Assumption \ref{assum:Q2}, we
can show that the algorithms converge at a global sublinear rate (with
a linear rate in the early stages) for the case of convex $F$,
provided that the final value of $H_k$ for each iteration $k$ of
Algorithms~\ref{alg:QuasiArmijo} and \ref{alg:QuasiModifyH} is
positive semidefinite.

\begin{theorem}
\label{thm:sublinear0}
Assume that $f$ is convex, \modify{Assumptions \ref{assum:general} and \ref{assum:Q2} hold}, $H_k
\succeq 0$ for all $k$, and there is some $\eta \in [0,1)$ such that
  the approximate solution $d^k$ of \eqref{eq:quadratic} satisfies
  \eqref{eq:approx} for all $k$.  Then the following claims for
  Algorithm \ref{alg:QuasiArmijo} are true.
\begin{enumerate}
\item When $F(x^k) - F^* \geq (x^k - P_\Omega(x^k))^T H_k (x^k -
  P_\Omega(x^k))$, we have a linear improvement of the objective error
  at iteration $k$, that is,
\begin{align}
F\left(x^{k+1}\right) - F^*
\leq \left( 1 - \frac{\left( 1 - \eta
	\right) \gamma \alpha_k}{2}\right) \left(F\left( x^k \right) -
	F^*\right).
	\label{eq:esl0}
\end{align}
\item
  For any $k \geq k_0$, where $k_0 \coloneqq \arg\min \{ k: F(x^k) -
  F^* < M R_0^2\}$, we have
\begin{align}
F\left(x^k\right) - F^* &\leq \frac{2 M R_0^2}{\gamma (1 -
	\eta) \sum_{t = k_0}^{k-1} \alpha_t + 2},
	\label{eq:sublinear0}
\end{align}
suggesting sublinear convergence of the objective error.
If there exists $\bar \alpha > 0$ such that $\alpha_k
\geq \bar\alpha$ for all $k$, we have
\begin{equation}
k_0 \leq \max\left\{0, 1 + \frac{2}{\gamma \left( 1 - \eta \right)
	\bar\alpha } \log \frac{F\left( x^0 \right) - F^*}{MR_0^2} \right\}.
\label{eq:k0}
\end{equation}
\end{enumerate}

For Algorithm \ref{alg:QuasiModifyH} under the condition
\eqref{eq:H0}, the above results still hold, with $\bar \alpha = 1$,
$\alpha_k \equiv 1$ for all $k$, and $M$ replaced by $\tilde
M_1(\eta)$ defined in \eqref{eq:m1} for Variant 1, and $\tilde
M_2(\eta)$ defined in \eqref{eq:m2} for Variant 2.
\end{theorem}
\begin{proof}
  Denoting $\delta_k \coloneqq F(x^k) - F^*$, we have for
  Algorithm~\ref{alg:QuasiArmijo} that the sufficient decrease
  condition \eqref{eq:armijo} together with $H_k \succeq 0$ imply that
\begin{equation}
\label{eq:Q_F}
\delta_{k+1} - \delta_k \leq \alpha_k \gamma \Delta_k
= \alpha_k \gamma \left(Q_k\left(d^k\right) - \frac12 \left( d^k
	\right)^T H_k d^k\right)
\leq \alpha_k \gamma Q_k\left(d^k\right).
\end{equation}
By defining
\begin{equation*}
	A_k \coloneqq \left(x^k - P_\Omega\left( x^k \right)\right)^T H_k
	\left(x^k - P_\Omega\left( x^k \right)\right), \quad
	A\coloneqq M R_0^2,
\end{equation*}
(note that $A_k \le A$ follows from \eqref{eq:A})
and using \eqref{eq:approx}, \eqref{eq:Q_F}, and \eqref{eq:Qlambda},
we obtain 
\begin{equation}
\delta_{k+1} - \delta_k \leq \alpha_k\gamma \left(1 -
	\eta\right)\left(-\lambda_k \delta_k + \frac{A_k\lambda_k^2}{2}
	\right),\quad \forall \lambda_k\in[0,1].
\label{eq:tmp3}
\end{equation}
We note that \eqref{eq:tmp3} satisfies \eqref{eq:seq} with
\begin{equation*}
		c_k = \alpha_k\gamma \left(1 - \eta\right).
\end{equation*}
The results now follow directly from Lemma~\ref{lemma:seq}.

For Algorithm \ref{alg:QuasiModifyH},
from \eqref{eq:decrease} and
\eqref{eq:approx}, we get that for any $k\geq 0$,
\begin{equation}
\delta_{k+1} - \delta_k \leq \gamma \left(1 - \eta\right)Q^*_k,
\label{eq:Q_F2}
\end{equation}
and the remainder of the proof follows the above procedure starting
from the right-hand side of \eqref{eq:Q_F} with $\alpha_k \equiv 1$.
\qed
\end{proof}

The conditions of Parts 1 and 2 of Theorem~\ref{thm:sublinear0} bear
further consideration.  When the regularization term $\psi$ is not
present in $F$, and $M$ is a global bound on the norm of the true
Hessian $\nabla^2 f(x)$, the condition in Part 2 of
Theorem~\ref{thm:sublinear0} is satisfied for $k_0=0$, since $f(x^0) -
f^* \le \frac12 M \|x^0 - P_\Omega (x^0) \|^2 \le \frac12
MR_0^2$. Under these circumstances, the linear convergence result of
Part 1 may appear not to be interesting. We note, however, that the
contribution from $\psi$ may make a significant difference in the
general case (in particular, it may result in $F(x^0)-F^* > MR_0^2$)
and, moreover, a choice of $H_k$ with $\| H_k \|$ significantly less
than $M$ may result in the condition of Part 1 being satisfied
intermittently during the computation. In particular, Part 1 lends some
support to the empirical observation of rapid convergence on the early
stages of the algorithms, as we discuss further below.
Note that \cite[Theorem~4]{Nes13a} suggests that when the algorithm is
exact proximal gradient, we get $F(x^k) - F^* \leq M R_0^2$ for all $k
\geq 1$, but this is not always the case when a different $H$ is
picked or when \eqref{eq:quadratic} is solved only approximately.


By combining Theorem~\ref{thm:sublinear0} with
Lemma~\ref{lemma:delta_d} and Corollary~\ref{lemma:delta_d2} (which
yield lower bounds on $\alpha_k$), we obtain the following results for
Algorithm~\ref{alg:QuasiArmijo}. 
\begin{corollary}
\label{cor:QorH}
Assume the conditions of Theorem \ref{thm:sublinear0} are all
satisfied. Then we have the following.
\begin{enumerate}
\item If there exists $\sigma > 0$ such that $\lambdamin(H_k) \ge
  \sigma$ for all $k$, then \eqref{eq:esl0} becomes
\begin{equation}
	\frac{F\left(x^{k+1}\right) - F^*}{F\left( x^k \right) - F^*}
	\leq 1 - \frac{\gamma}{2}\min\left\{\left(1 - \eta \right)
	, \frac{2\left(1 - \sqrt\eta\right)\beta(1 -
	\gamma)\sigma }{L}\right\},
	\label{eq:esl0sigma}
\end{equation}
\eqref{eq:sublinear0} becomes
\begin{equation*}
	F\left(x^k\right) - F^* \leq \frac{2M R_0^2}{\gamma (k -
		k_0) \min\left\{ 1 - \eta, \frac{2 \left( 1 - \sqrt{\eta}
	\right)\beta \left( 1 - \gamma \right)\sigma }{L} \right\} + 2},
\end{equation*}
and \eqref{eq:k0} becomes
\begin{equation*}
	k_0 < 1+ \frac{2}{\gamma}\max\left\{0, \log \frac{F\left( x^0
	\right) - F^*}{M R_0^2} \right\} \cdot \max \left\{\frac{1}{(1 - \eta)},
	\frac{L}{2 (1 - \sqrt{\eta})\beta \left( 1 - \gamma \right)
\sigma} \right\}.
\end{equation*}
\item If $Q_k$ is $\sigma$-strongly convex and $H_k \succeq 0$ for all
	$k$, then
\eqref{eq:esl0} becomes
\begin{align*}
	\frac{F\left(x^{k+1}\right) - F^*}{F\left( x^k \right) - F^*}
	\leq 1 - \frac{\gamma}{2}\min\left\{ 1 - \eta
	,  \frac{\left(1 - \sqrt\eta\right)^2\beta(1 -
	\gamma)\sigma }{L}\right\},
\end{align*}
\eqref{eq:sublinear0} becomes
\begin{equation*}
	F\left(x^k\right) - F^* \leq \frac{2M R_0^2}{\gamma (k -
	k_0)\min\left\{ 1 - \eta, \frac{(1 -
		\sqrt{\eta})^2\beta \left( 1 - \gamma \right)
	\sigma}{L}\right\} + 2},
\end{equation*}
and \eqref{eq:k0} becomes
\begin{equation*}
	k_0 < 1 + \frac{2}{\gamma}\max\left\{0, \log \frac{F\left( x^0
	\right) - F^*}{M R_0^2} \right\}\max \left\{\frac{1}{(1 - \eta)},
	\frac{L}{(1 - \sqrt{\eta})^2\beta \left( 1 - \gamma \right)
\sigma} \right\}.
\end{equation*}
\end{enumerate}
\end{corollary}

We make some remarks on the results above.
\begin{remark}
For any $\eta\in[0,1)$, we have
\begin{equation*}
\frac{1}{2(1 - \sqrt{\eta})} < \frac{1}{1 - \eta}  < \frac{1}{(1 - \sqrt{\eta})^2}.
\end{equation*}
Therefore, Algorithm~\ref{alg:QuasiArmijo} with positive definite
$H_k$ has better dependency on $\eta$ than the case in which we set
$\lambdamin(H_k) = 0$ and rely on $\psi$ to make $Q_k$ strongly
convex. If $\psi$ is strongly convex, we can move some of its
curvature to $H_k$ without changing the subproblems
\eqref{eq:quadratic}. This strategy may require us to increase $M$,
but this has only a slight effect on the bounds in
Corollary~\ref{cor:QorH}. These bounds give good reasons to capture
the curvature of $Q_k$ in the Hessian $H_k$ alone, so henceforth we
focus our discussion on this case.
\end{remark}

\begin{remark}
For Algorithm \ref{alg:QuasiModifyH}, when we use the bounds
\eqref{eq:m1} and \eqref{eq:m2} for $M$ in \eqref{eq:R0}, the
dependency of the global complexity on $\eta$ becomes
\begin{equation*}
	\max\left\{\frac{1}{1 - \eta}, 
	\frac{1}{\left(2 - \gamma \left(1 - \sqrt{\eta}\right)\right)(1 - \sqrt{\eta})}
\right\}
	\leq
	\max\left\{\frac{1}{1 - \eta},
	\frac{1}{(2 - \gamma) (1 - \sqrt{\eta})}\right\},
\end{equation*}
  This result is slightly worse than that of using
  positive definite $H$ in Algorithm \ref{alg:QuasiArmijo} if we compare
  the second part in the max operation.
\end{remark}

\begin{remark}
The bound in \eqref{eq:A} is not tight for general $H$, unless
$H_k \equiv MI$, as in standard prox-gradient methods.  This
observation gives further intuition for why second-order methods tend
to perform well even though their iteration complexities (which are
based on the bound \eqref{eq:A}) tend to be worse than first-order
methods.
Moreover, when $H_k$ incorporates curvature information for $f$, step
sizes $\alpha_k$ are often much larger than the worst-case bounds that
are used in Corollary~\ref{cor:QorH}. Theorem~\ref{thm:sublinear0},
which shows how the convergence rates are related directly to the
$\alpha_k$, would give tighter bounds in such cases.
Line search on $H_k$ in Algorithm~\ref{alg:QuasiModifyH} does not
improve the rate directly, but we note that using $H_k$ with smaller
norm whenever possible gives more chances of switching to the
intermittent linear rate \eqref{eq:esl0}.

\end{remark}

Part 1 of Theorem~\ref{thm:sublinear0} also explains why solving the
{\em subproblem} \eqref{eq:quadratic} approximately can save the
running time significantly, since because of fast early convergence
rate, a solution of moderate accuracy can be attained relatively
quickly.

\subsubsection{Linear Convergence for Optimal Set Strongly Convex Functions}


We now consider problems that satisfy the
$\mu$-optimal-set-strong-convexity condition \eqref{eq:strong} for
some $\mu > 0$, and show that our algorithms have a global linear
convergence property.

\begin{theorem} \label{th:7}
If \modify{Assumption \ref{assum:general} holds}, $f$ is convex, $F$
is $\mu$-optimal-set-strongly convex for some $\mu > 0$, there is some
$\eta \in [0,1)$ such that at every iteration of
	Algorithm~\ref{alg:QuasiArmijo}, the approximate solution $d$ of
	\eqref{eq:quadratic} satisfies \eqref{eq:approx}, and
\begin{equation}
\label{eq:H}
\sigma I \preceq H_k \preceq MI, \quad \text{for some }M\geq \sigma
>0,\quad \forall k.
\end{equation}
Then for $k=0,1,2,\dotsc$, we have
\begin{subequations}
\begin{align}
	\label{eq:qlinear0}
	&~\frac{F\left(x^{k+1}\right) - F^*}{F\left( x^k \right) - F^*}
	\leq 1 - \frac{\alpha_k \gamma \left(1 - \eta\right)\mu}{\mu
	+ \|H_k\|} \\
	\leq&~ 1 - \frac{\gamma \mu}{\mu + M} \min\left\{ \left( 1 - \eta
		\right), \frac{2 \left( 1 - \sqrt{\eta} \right)\beta \left( 1
	- \gamma  \right) \sigma}{L}\right\}.
\label{eq:qlinear}
\end{align}
\end{subequations}
Moreover, on iterates $k$ for which $F(x^k) - F^* \geq (x^k -
P_\Omega(x^k))^T H_k (x^k -P_\Omega(x^k))$, these per-iteration
contraction rates can be replaced by the faster rates \eqref{eq:esl0}
and \eqref{eq:esl0sigma}.
\end{theorem}
\begin{proof}
  By rearranging \eqref{eq:Q_F}, we have
  \begin{subequations}
\begin{align}
F\left(x^{k+1}\right) - F^* &\leq F\left(x^k\right) - F^* + \alpha_k
	\gamma Q_k\left(d^k\right)\nonumber\\
&\leq F\left(x^k\right) - F^* + \alpha_k \gamma \left(1 -
	\eta\right) Q_k^*
\label{eq:intermediate.1}\\
&\leq F\left( x^k \right) - F^* - \alpha_k \gamma \left( 1 - \eta
	\right) \frac{\mu}{\mu + \|H_k\|} \left(F\left(x^k\right) -
	F^*\right)
\label{eq:intermediate.2}\\
\nonumber
&= \left(1 - \alpha_k \gamma \left(1 - \eta\right)
	\frac{\mu}{\mu + \|H_k\|}\right) \left(F\left(x^k\right)
	- F^*\right),
\end{align}
  \end{subequations}
where in \eqref{eq:intermediate.1} we used the inexactness condition
\eqref{eq:approx} and in \eqref{eq:intermediate.2} we used
\eqref{eq:Qstrong}.  Using the result in
Corollary~\ref{lemma:delta_d2} to lower-bound $\alpha_k$, we obtain
\eqref{eq:qlinear}.

To show that the part for the early fast rate in \eqref{eq:esl0} and
\eqref{eq:esl0sigma} can be applied, we show that Assumption
\ref{assum:Q2} holds. Then because $f$ is assumed to be convex as well
here, Theorem \ref{thm:sublinear0} and Corollary \ref{cor:QorH} apply as well.
Consider \eqref{eq:strong}, by rearranging the terms, we get
\begin{align}
	\nonumber
	\lambda \left( F(x) - F^* \right) &\geq
	\frac{\mu \lambda (1 - \lambda)}{2}\left\| x - P_{\Omega}\left( x
	\right) \right\|^2 +
	F\left( \lambda x + \left( 1 - \lambda \right)P_{\Omega}\left(
	x \right) \right) - F^*\\
	&\geq
	\frac{\mu \lambda (1 - \lambda)}{2}\left\| x - P_{\Omega}\left( x
	\right) \right\|^2
	, \quad \forall \lambda \in [0,1],
	\label{eq:inter}
\end{align}
as $F\left( \lambda x + \left( 1 - \lambda \right) P_{\Omega}\left(
x \right) \right) \geq F^*$ from optimality.
By dividing both sides of \eqref{eq:inter} by $\lambda$ and letting
$\lambda \rightarrow 0$, we get the bound
\begin{equation}
	F(x^0) - F^* \geq F(x) - F^* \geq \frac{\sigma}{2} \|x -
	P_{\Omega}(x)\|, \forall x: F(x) \leq F(x^0),
\end{equation}
validating Assumption \ref{assum:Q2}.
\qed
\end{proof}

\modify{Note that the parameter $\mu$ in the theorem above is decided
	by the problem and cannot be changed, while $\sigma$ can be
altered according to the algorithm choice.}
We have a similar result for Algorithm~\ref{alg:QuasiModifyH}.

\begin{theorem} \label{th:8}
If \modify{Assumption \ref{assum:general} holds}, $f$ is
convex,  $F$ is $\mu$-optimal-set-strongly convex
for some $\mu > 0$, there exists some $\eta \in [0,1)$ such that at
	every iteration of Algorithm~\ref{alg:QuasiModifyH},
the approximate solution $d$ of \eqref{eq:quadratic} satisfies
\eqref{eq:approx}, and the conditions for $H^0_k$ in
Lemma~\ref{lemma:Hbound} are satisfied for all $k$.  Then we have
\begin{equation}
	\frac{F\left(x^{k+1}\right) - F^*}{F\left( x^{k} \right) - F^*
}\leq 1 - \gamma \frac{\mu\left( 1 - \eta \right)}{\mu +
	\|H_k\|}
	, \quad k=0,1,2,\dotsc,
	\label{eq:Hlinear}
\end{equation}
and the right-hand side of \eqref{eq:Hlinear} can be further bounded
by
\begin{equation}
1 - \gamma \frac{\mu\left(1 - \eta\right)}{\mu + \tilde{M}_1(\eta)} \quad \mbox{ and } \quad
1 - \gamma \frac{\mu\left(1 - \eta\right)}{\mu + \tilde{M}_2(\eta)}
\label{eq:iters}
\end{equation}
for Variant 1 and Variant 2, respectively, where $\tilde{M}_1(\eta)$ and
$\tilde{M}_2(\eta)$ are defined in Lemma~\ref{lemma:Hbound}.
Moreover, when $F(x^k) - F^* \geq (x^k - P_\Omega(x^k))^T H_k (x^k
-P_\Omega(x^k))$,
the faster rate \eqref{eq:esl0} (with $\alpha_k \equiv 1$ and the
modification for Algorithm \ref{alg:QuasiModifyH} mentioned in Theorem
\ref{thm:sublinear0})
can be used to replace \eqref{eq:Hlinear}.
\end{theorem}
\begin{proof}
From \eqref{eq:Qstrong} and \eqref{eq:Q_F2}, we have
\begin{align*}
F\left(x^{k+1}\right) - F^*
&\leq F\left(x^k\right) - F^* + \gamma Q_k\left(d^k\right) \\
&\leq F\left(x^k\right) - F^* + \gamma \left(1 - \eta\right) Q_k^* \\
&\leq \left(1 - \gamma \frac{\mu}{\mu + \|H_k\|} \left(1 -
	\eta\right)\right)\left(F\left(x^k\right) - F^*\right),
\end{align*}
proving \eqref{eq:Hlinear}.
From Lemma~\ref{lemma:Hbound}, we ensure that $\|H_k\|$ is
upper-bounded by $\tilde M_1(\eta)$ and $\tilde M_2(\eta)$ for the two
variants respectively,
leading to \eqref{eq:iters}.
The statement concerning  \eqref{eq:esl0} follows from the same reasoning as in the proof
for Theorem \ref{th:7}.
	\qed
\end{proof}

By reasoning with the extreme eigenvalues of $H_k$, we can see that
the convergence rates still depend on the conditioning of $f$. For
Algorithm~\ref{alg:QuasiArmijo}, if we select $M \leq L$, then
backtracking may be necessary, and the bound \eqref{eq:qlinear} (in
which a factor $\mu/L$ appears) is germane.  This same factor appears
in both \eqref{eq:qlinear0} and \eqref{eq:qlinear} when $M>L$. Often,
however, the backtracking line search chooses a value of $\alpha_k$
that is not much less than $1$, which is why we believe that the
bounds \eqref{eq:esl0}, \eqref{eq:sublinear0}, and \eqref{eq:qlinear0}
(which depend explicitly on $\alpha_k$) have some value in revealing
the actual performance of the algorithm.  Similar comments apply to
Algorithm~\ref{alg:QuasiModifyH}, because \eqref{eq:decrease} may be
satisfied with $\|H_k\|$ much smaller than the bounds for properly
chosen $H^0_k$. 



In the interesting case in which we choose $H_k \equiv L I$ and $\eta
= 0$, we have $m_0 = \|H_k\| = L$ in Algorithm \ref{alg:QuasiModifyH},
and modification of $H_k$ is not needed, since \eqref{eq:decrease}
always holds for $\gamma = 1$. The bound \eqref{eq:sublinear0} becomes
$(F(x^k) - F^*) \leq 2 L R_0^2 / (k + 2)$, which matches the known
convergence rates of proximal gradient \citep{Nes13a} and gradient
descent \citep{Nes04a}.  The global linear rate in Theorem~\ref{th:8}
also matches that of existing proximal gradient analysis for strongly
convex problems, but the intermittent linear rate \eqref{eq:esl0} that
applies to both cases is new.  For the case of accelerated proximal
gradient covered in \cite{Nes13a}, although not covered directly by
our framework studied in this work,
one can combine our algorithm and analysis with the Catalyst framework
\citep{LinMH15a} to obtain similar accelerated rates for both the
strongly convex and the general convex cases.

\subsubsection{Sublinear Convergence of the First-order Optimality
Condition for Nonconvex Problems}

We consider now the case of nonconvex $F$.  In this situation,
Lemma~\ref{lemma:Q} cannot be used, so we consider other properties of
$Q$.  We can no longer guarantee the convergence of the objective
value to the global minimum.  Instead, we consider the norm of the
exact solution of the subproblem as the indicator of closeness to
the first-order optimality condition $0 \in \partial F(x)$ for
\eqref{eq:f} (see, for example, \cite[(14.2.16)]{Fle87a}).
In particular, it is known that $0 \in \partial F(x)$ if and only if
\begin{equation}
	\label{eq:1oa}
0 = \arg\min_d \, Q_I^x(d)  = \arg\min_d \,   \nabla f\left( x \right)^T d + \frac12
d^Td  + \psi\left( x+d \right) - \psi\left( x \right).
\end{equation}
This is a consequence of the following lemma.
\begin{lemma}
\label{lemma:proxg}
Given any $H \succ 0$, and $Q_H^x$ defined as in \eqref{eq:quadratic},
the following are true.
\begin{enumerate}
\item
	A point $x$ satisfies the first-order optimality condition
	$0 \in \partial F(x)$ if and only if
	\begin{equation*}
		0 = \arg\min_{d} \, Q^x_H(d).
	\end{equation*}
\item
	For any $x$, defining $d^*$ to be the minimizer of
        $Q_H^x(\cdot)$, we have
	\begin{equation}
	Q_H^x(d^*)  \leq - \frac12 {\lambdamin\left( H \right)} \|d^*\|^2.
	\label{eq:opt}
	\end{equation}
\end{enumerate}
\end{lemma}
\begin{proof}
  Part 1 is well known. For Part 2, we have from the
  optimality conditions for $d^*$ that $-\nabla f(x) - Hd^* \in
  \partial \psi(x+d^*)$. By convexity of $\psi$, we thus have
  \[
  \psi(x) \ge \psi(x+d^*) + (d^*)^T(\nabla f(x) + Hd^*)
  \;\; \Rightarrow \;\; 0 \ge Q_H^x(d^*) + \frac12 (d^*)^THd^*,
  \]
  from which the result follows.
	\qed
\end{proof}

As in \eqref{eq:1oa}, we consider the following measure of closeness to
a stationary point:
\begin{equation}
    G_k \coloneqq \arg\min_d \,
  Q_I^{x^k}(d).
\label{eq:Gk}
\end{equation}
We show that the minimum value of the norm of this measure over the
first $k$ iterations converges to zero at a sublinear rate of
$O(1/\sqrt{k})$.
The first step is to show that the minimum of $|Q_k|$ converges at a
$O(1/k)$ rate.
\begin{lemma}
\label{lemma:Qbound}
Assume that there is an $\eta \in [0,1)$ such that \eqref{eq:approx}
  is satisfied at all iterations. For Algorithm~\ref{alg:QuasiArmijo},
  if \modify{Assumption \ref{assum:general} holds} and $H_k \succeq \sigma I$ for some $\sigma>0$ and all $k$, we have
\begin{align}
\min_{0 \leq t \leq k} \left|Q_t\left( d^t \right)\right|
\leq \frac{F\left( x^0 \right) -F^*}{\gamma\left(k+1\right)\min_{0
\leq t \leq k} \alpha_t}
\leq \frac{F\left( x^0 \right) -F^*}{\gamma\left(k+1\right)}
\max\left\{1, \frac{(1 + \sqrt{\eta})L}{2 \beta\left(1 -
\gamma\right)\sigma}
\right\}.
\label{eq:Qopt}
\end{align}
For Algorithm~\ref{alg:QuasiModifyH} (requires $H_k^0 \succ 0$ for
the first variant),
we have
\begin{equation*}
\min_{0 \leq t \leq k} \left|Q_t\left(d^t\right)\right| \leq
\frac{F\left(x^0\right) - F^*}{\gamma\left(k+1\right)}.
\end{equation*}
\end{lemma}
\begin{proof}
From \eqref{eq:Q_F}, we have that for any $k \geq 0$,
\begin{equation}
F^* - F\left(x^0\right) \leq F\left(x^{k+1}\right) - F\left(x^0\right)
\leq \gamma \sum_{t=0}^k \alpha_t Q_t \left(d^t\right)
\leq \gamma \min_{0 \leq t \leq k} \alpha_t \sum_{t=0}^k Q_t\left(
d^t \right).
\label{eq:Qbound}
\end{equation}
From Corollary~\ref{lemma:delta_d2}, we have that $\alpha_t$ for all
$t$ is lower bounded by a positive value.  Therefore, using $\left|
Q_t\left(d^t\right)\right| = - Q_t\left(d^t\right)$ for all $t$, we
obtain
\begin{equation*}
\min_{0 \leq t \leq k} \left| Q_t\left(d^t\right)\right| 
\leq -\frac{1}{k+1} \sum_{t=0}^k Q_t\left(d^t\right)
\leq \frac{F\left(x^0\right) - F^*}{\gamma \left(k+1\right)\min_{0\leq
t \leq k}\alpha_t }.
\end{equation*}
Substituting the lower bound for $\alpha$ from
Corollary~\ref{lemma:delta_d2} gives the desired result
\eqref{eq:Qopt}.
The result for Algorithms~\ref{alg:QuasiModifyH}
follows from the same reasoning applied to \eqref{eq:decrease}.
	\qed
\end{proof}

The following lemma is from \cite{TseY09a}. (Its proof is omitted.)
\begin{lemma}[{\cite[Lemma 3]{TseY09a}}]
\label{lemma:Gk}
Given $H_k$ satisfying \eqref{eq:H} for all $k$, we have
\begin{equation*}
\left\|G_k\right\| \leq \frac{1 + \frac{1}{\sigma} + \sqrt{1 - 2
	\frac{1}{M} + \frac{1}{\sigma^2}}}{2}
	M \left\|d^{k*} \right\|,
\end{equation*}
where
\begin{equation*}
d^{k*} \coloneqq \arg\min Q_k.
\end{equation*}
\end{lemma}
We are now ready to show the convergence of $\|G_k\|$.
\begin{corollary}
\label{cor:Gk}
Assume that \eqref{eq:approx} holds at all iterations for some $\eta
\in [0,1)$ and that \modify{Assumption \ref{assum:general} holds}. Let $\tilde{M}_1(\eta)$ and $\tilde{M}_2(\eta)$ be as
  defined in Lemma~\ref{lemma:Hbound}. For
  Algorithm~\ref{alg:QuasiArmijo}, suppose that $H_k$ satisfies
  \eqref{eq:H} for all $k \ge 0$. We then have for all
  $k=0,1,,2,\dotsc$ that
\begin{align*}
\min_{0 \leq t \leq k }\left\|G_t\right\|^2
\leq& \frac{F\left( x^0
	\right) - F^*}{\gamma \left( k+1 \right)} \, \frac{M^2\left(1 +
	\frac{1}{\sigma} + \sqrt{1 - \frac{2}{M} +
	\frac{1}{\sigma^2}}\right)^2}{2 (1 - \eta)\sigma\min_{0\leq t \leq k} \alpha_t} \\
\leq& \frac{F\left( x^0
	\right) - F^*}{\gamma \left( k+1 \right)} \, \frac{M^2\left(1 +
	\frac{1}{\sigma} + \sqrt{1 - \frac{2}{M} +
	  \frac{1}{\sigma^2}}\right)^2}{2 \sigma} \\
& \quad\quad \max\left\{\frac{1}{1 -
	\eta}, \frac{L}{2\left(1 - \sqrt{\eta}\right) \left( 1 -
	\gamma\right) \sigma \beta}\right\}.
\end{align*}
For Algorithm~\ref{alg:QuasiModifyH}, if the initial $H_k^0$ satisfies
$M_0 I \succeq H_k^0 \succeq m_0 I$ with $M_0 \ge m_0 >0$ then for
Variant 1 we have:
\begin{align*}
\min_{0 \leq t \leq k }\left\|G_t\right\|^2
\leq \frac{F\left( x^0
\right)) - F^*}{\gamma
\left( (k+1 \right))} \frac{\tilde{M}_1(\eta)^2\left(1 + \frac{1}{m_0} +
\sqrt{1 - \frac{2}{\tilde{M}_1(\eta)}+ \frac{1}{m_0^2}}\right)^2}{2 \left(1 -
\eta\right)  m_0}.
\end{align*}
For Variant 2, we have under the same assumptions on $H_k^0$ that the
same bound is satisfied,\footnote{We could instead require only $H_k^0
\succeq 0$ and start with $H_k + I$ instead.} with $\tilde{M}_1(\eta)$
replaced by $\tilde{M}_2(\eta)$.
\end{corollary}
\begin{proof}
Let $\bar{k} \coloneqq \arg \min_{0\leq t \leq k}\left|Q_t(d^t)\right|$, the
condition \eqref{eq:approx} and Lemmas~\ref{lemma:proxg} and \ref{lemma:Gk}
imply
\begin{align}
	\nonumber
-Q_{\bar{k}}\left(d^{\bar{k}}\right)
&\geq -\left(1 - \eta\right) Q_{\bar k}^*\\
\nonumber
&\geq \frac{\sigma\left(1 - \eta\right)}{2} \left\|d^{\bar{k} *}\right\|^2\\
&\geq \frac{2 \sigma\left(1 - \eta\right)}{M^2 \left(1 +
	\frac{1}{\sigma} + \sqrt{1 - \frac{2}{M} +
	\frac{1}{\sigma^2}}\right)^2}\left\|G_{\bar k}\right\|^2.
\label{eq:Gkbound}
\end{align}
Finally, we note that $\|G_{\bar k}\| \geq \min_{0\leq t \leq k}
\|G_t\|$. The proof is finished by combining \eqref{eq:Gkbound} with
Lemma~\ref{lemma:Qbound}.
	\qed
\end{proof}

If we replace the definition of $G_k$ in \eqref{eq:Gk} by the solution
of \eqref{eq:quadratic}, the inequality in Lemma~\ref{lemma:Gk} is not
needed.  In particular, when we use the proximal gradient algorithm
with $H_k = LI$ and $\eta = 0$ (so that \eqref{eq:decrease} holds
with $\gamma = 1$, and $M = L$)
we obtain a bound of $2(F(x^0) - F^*) /
(L(k+1))$ on $\|d^k\|^2$,
matching the result shown in \cite{Nes13a,DruL16a}.

\subsubsection{Comparison Among Different Approaches}

Algorithms~\ref{alg:QuasiArmijo} and \ref{alg:QuasiModifyH} both
require evaluation of the function $F$ for each choice of the
parameter $\alpha_k$, to check whether the decrease conditions
\eqref{eq:armijo} and \eqref{eq:decrease} (respectively) are
satisfied. The difference is that Algorithm~\ref{alg:QuasiModifyH} may
also require solution of the subproblem \eqref{eq:quadratic} for each
$\alpha_k$. This additional computation comes with two potential
benefits.  First, the second variant of
Algorithm~\ref{alg:QuasiModifyH} allows the initial choice of
approximate Hessian $H^0_k$ to be indefinite, although the final value
$H_k$ at each iteration needs to be positive semidefinite for our
analysis to hold.  (There is a close analogy here to trust-region
methods for nonconvex smooth optimization, where an indefinite Hessian
is adjusted to be positive semidefinite in the process of solving the
trust-region subproblem.)  Second, because full steps are always taken
in Algorithm~\ref{alg:QuasiModifyH}, any structure induced in the
iterates $x^k$ by the regularizer $\psi$ (such as sparsity) will be
preserved. This fact in turn may lead to faster convergence, as the
algorithm will effectively be working in a low-dimensional subspace.

\section{Choosing $H_k$}
\label{sec:H}

Here we discuss some ways to choose $H_k$ so that the algorithms are
well defined and practical, and our convergence theory can be applied.

When $H_k$ are chosen to be positive multiples of identity ($H_k =
\zeta_k I$, say), our algorithms reduce to variants of proximal
gradient.  If we set $\zeta_k  \geq L$, then the unit step size is
always accepted even if the problem is not solved exactly, because
$Q_k(d^k)$ is an upper bound of $F(x^k) - F(x^k+d^k)$.  When $L$ is
not known in advance, adaptive strategies can be used to find it.  For
Algorithm~\ref{alg:QuasiModifyH}, we
could define $\zeta_k^0$ (such that $H_k^0 = \zeta_k^0 I$) to be the
final value $\zeta_{k-1}$ from the previous iteration, possibly
choosing a smaller value at some iterations to avoid being too
conservative. For Algorithm~\ref{alg:QuasiArmijo}, we could increase
$\zeta_k^0$ over $\zeta_{k-1}$ if backtracking was necessary at
iteration $k-1$, and shrink it when a unit stepsize sufficed for
several successive iterations.


The proximal Newton approach of setting $H_k = \nabla^2 f(x^k)$ is a
common choice in the convex case \citep{LeeSS14a}, where we can
guarantee that $H_k$ is at least positive semidefinite.  In
\cite{LeeSS14a}, it is shown that in some neighborhood of the optimum,
when $d^k$ is the exact solution of \eqref{eq:quadratic}, then unit
step size is always taken, and superlinear or quadratic convergence to
the optimum ensues.  (A global complexity condition is not required
for this result.)  Generally, however, indefiniteness in $\nabla^2
f(x^k)$ may lead to the search direction $d^k$ not being a descent
direction, and the backtracking line search will not terminate in this
situation. (Our convergence results for
Algorithm~\ref{alg:QuasiArmijo} do not apply in the case of $H_k$
indefinite.)
A common fix is to use damping, setting $H_k= \nabla^2 f(x^k)+ \zeta_k
I$, for some $\zeta_k \ge 0$ that at least ensures positive
definiteness of $H_k$. Strategies for choosing $\zeta_k$ adaptively
have been the subject of much research in the context of smooth
minimization, for example, in trust-region methods and the
Levenberg-Marquardt method for nonlinear least squares (see
\cite{NocW06a}). Variant 2 of our Algorithm~\ref{alg:QuasiModifyH}
uses this strategy. It is desirable to ensure that $\zeta_k \to 0$ as
the iterates approach a solution at which local convexity holds, to
ensure rapid local convergence.



An L-BFGS approximation of $\nabla^2 f(x^k)$ could also be used for
$H_k$. When $\psi$ is not present in \eqref{eq:f} and $f$ is strongly
convex, it is shown in \cite{LiuN89a} that this approach has global
linear convergence because the eigenvalues of $H_k$ are restricted to
a bounded positive interval. This proof can be extended to our
algorithms, when a convex $\psi$ is present in \eqref{eq:f}. When $f$
is not strongly convex, one can apply safeguards to the L-BFGS update
procedure (as described in \cite{LiF01a}) to ensure that the upper and
lower eigenvalues of $H_k$ are bounded uniformly away from zero.


Another interesting choice of $H_k$ is a block-diagonal approximation
of the Hessian, which (when $\psi$ can be partitioned accordingly)
allows the subproblem \eqref{eq:quadratic} to be solved in parallel
while still retaining some curvature information. Strategies like this
one are often used in distributed optimization for machine learning
problems (see, for example,
\anonymous{\cite{Yan13a,ZheXXZ17a}}{\cite{Yan13a,LeeC17a,ZheXXZ17a}}).

\section{Numerical Results}
\label{sec:exp}

We sketch some numerical simulations that support our theoretical
results.  We conduct experiments on two different problems:
$\ell_1$-regularized logistic regression, and the Lagrange dual
problem of $\ell_2$-regularized squared-hinge loss minimization. The
algorithms are implemented in C/C++.

\subsection{$\ell_1$-regularized Logistic Regression}
Given training data points $(a_i, b_i) \in \R^{n} \times \{-1,1\}$,
$i=1,\dotsc, l$, and a specified parameter $C>0$, we solve the
following convex problem
\begin{equation}
	\min_{x \in \R^n}\, C\sum_{i=1}^l \psi\left(1 + \exp\left(-b_i
	a_i^T x \right)\right) + \|x\|_1.
	\label{eq:l1lr}
\end{equation}
We define $H_k$ to be the limited-memory BFGS approximation
\citep{LiuN89a} based on the past 10 steps, with a safeguard mechanism
proposed in \cite{LiF01a} to ensure uniform boundedness of $H_k$.  The
subproblems \eqref{eq:quadratic} are solved with SpaRSA
\citep{WriNF08a}, a proximal-gradient method which, for bounded $H_k$,
converges globally at a linear rate.  We consider the publicly
available data sets listed in
Table~\ref{tbl:data},\footnote{Downloaded from
  \url{http://www.csie.ntu.edu.tw/~cjlin/libsvmtools/datasets/}.}  and
present empirical convergence results by showing the relative
objective error, defined as
\begin{equation}
	\label{eq:optf}
	\frac{F(x) - F^*}{F^*},
\end{equation}
where $F^*$ is the optimum, obtained approximately through running our
algorithm with long enough time.  For all variants of our framework,
we used parameters $\beta = 0.5$, and $\gamma = 10^{-4}$.  Further
details of our implementation are described in \cite{LeeLW18a}.

We use the two smaller data sets a9a and rcv1 to quantify the
relationship between accuracy of the subproblem solution and the
number of outer iterations.  We compare running SpaRSA with a fixed
number of iterations $T \in \{5,10,15,20,25,30\}$.
Figure~\ref{fig:conv} shows that, in all cases, the number of outer
iterations decreases monotonically as the (fixed) number of inner
iterations is increased. For $T \ge 15$, the degradation in number of
outer iterations resulting from less accurate solution of the
subproblems is modest, as our theory suggests.  We also observe the
initial fast linear rates in the early stages of the method that are
predicted by our theory, settling down to a slower linear rate on
later iterations, but with sudden drops of the objective, possibly as
a consequence of intermittent satisfaction of the condition in Part 1
of Theorem~\ref{thm:sublinear0}.


Next, we examine empirically the step size distribution for
Algorithm~\ref{alg:QuasiArmijo} and how often in
Algorithm~\ref{alg:QuasiModifyH} the matrix $H_k$ needs to be
modified. On both a9a and rcv1, the initial step estimate $\alpha=1$
is accepted on over 99.5\% of iterations in
Algorithm~\ref{alg:QuasiArmijo}, while in both variants of
Algorithm~\ref{alg:QuasiModifyH}, the initial choice of $H_k$ is used
without modification on over 99\% of iterations. These statistics hold
regardless of the value of $T$ (the number of inner iterations),
though in the case of Algorithm~\ref{alg:QuasiModifyH}, we see a faint
trend toward more adjustments for larger values of $T$. When
adjustments are needed, they never number more than $4$ at any one
iteration, except for a single case (a9a for Variant 1 of
Algorithm~\ref{alg:QuasiModifyH} with $T=5$) for which up to $8$
adjustments are needed.

We next compare our inexact method with an exact version, in which the
subproblems \eqref{eq:quadratic} are solved to near-optimality at each
iteration.
Since the three algorithms behave similarly, we use
Algorithm~\ref{alg:QuasiArmijo} as the representative for this
investigation.  We use a local cluster with 16 nodes for the two
larger data sets rcv1 and epsilon, while for the small data set a9a,
only one node is used.  Iteration counts and running time comparisons
are shown in Figure~\ref{fig:exact}.  The exact version requires fewer
iterations, as expected, but the inexact version requires only
modestly more iterations.  In terms of runtime, the inexact versions
with moderate amount of inner iterations (at least $30$) has the
advantage, due to the savings obtained by solving the subproblem
inexactly.

We note that the approach of gradually increasing the number of inner
iterations, suggested in \cite{SchT16a,GhaS16a}, produces good results
for this application, the number of iterations required being
comparable to those for the exact solver while the running time is
slightly faster than that of $T=30$ for epsilon and competitive with
it for the rest two data sets.

\begin{table}
	\centering
	\begin{tabular}{l|rrr}
		Data set & $l$ & $n$ & \#nonzeros\\
		\hline
		a9a & $32,561$ & $123$ & $451,592$\\
		rcv1\_test.binary & $677,399$ & $47,236$ & $49,556,258$\\
		epsilon & $400,000$ & $2,000$ & $800,000,000$
	\end{tabular}
	\caption{Properties of the Data Sets}
	\label{tbl:data}
\end{table}

\begin{figure}[tb]
\centering
\begin{subfigure}[b]{\linewidth}
\begin{tabular}{ccc}
	Algorithm \ref{alg:QuasiArmijo} & Variant 1 of Algorithm
	\ref{alg:QuasiModifyH} & Variant 2 of Algorithm
	\ref{alg:QuasiModifyH}\\
\includegraphics[width=.32\linewidth]{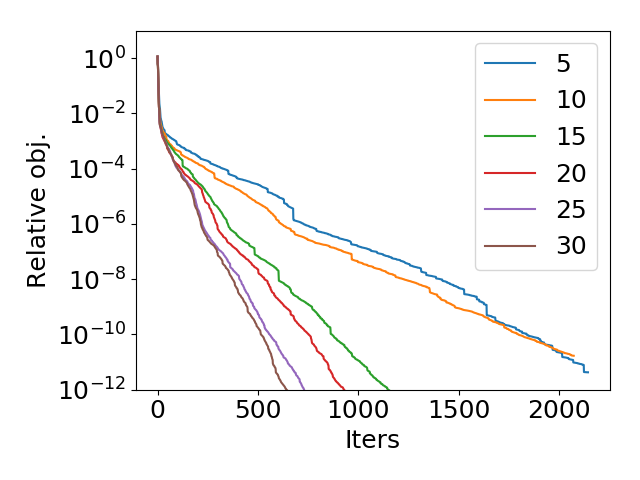}&
\includegraphics[width=.32\linewidth]{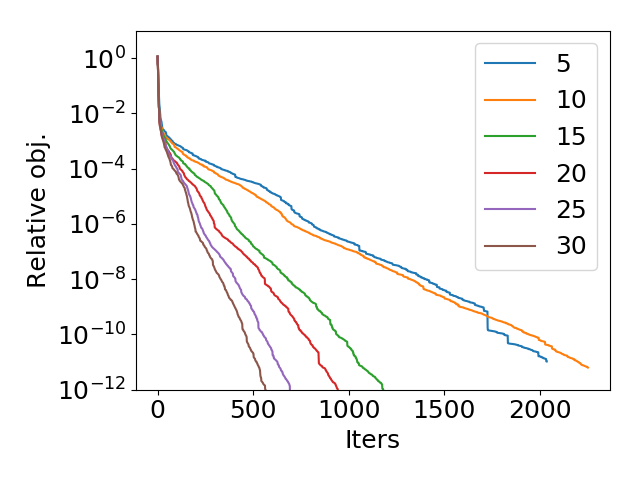}&
\includegraphics[width=.32\linewidth]{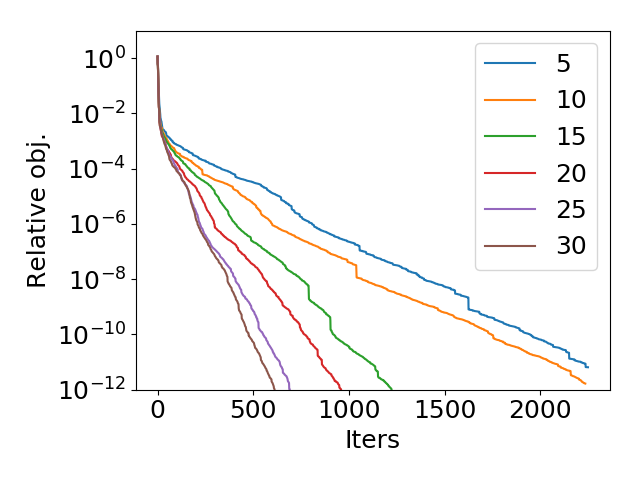}
\end{tabular}
\caption{a9a}
\end{subfigure}
\begin{subfigure}[b]{\linewidth}
\begin{tabular}{ccc}
\includegraphics[width=.32\linewidth]{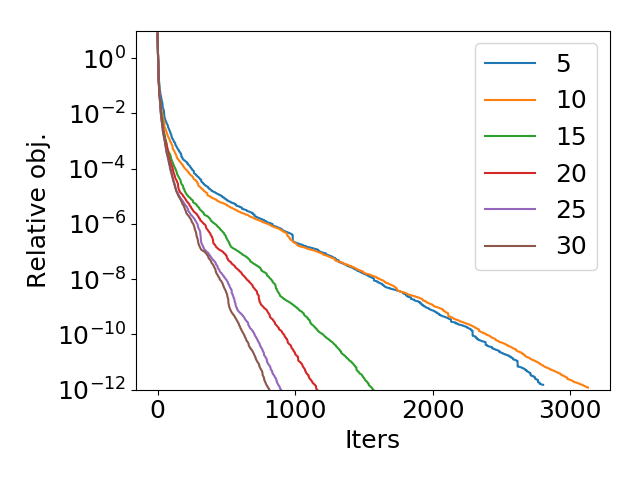}&
\includegraphics[width=.32\linewidth]{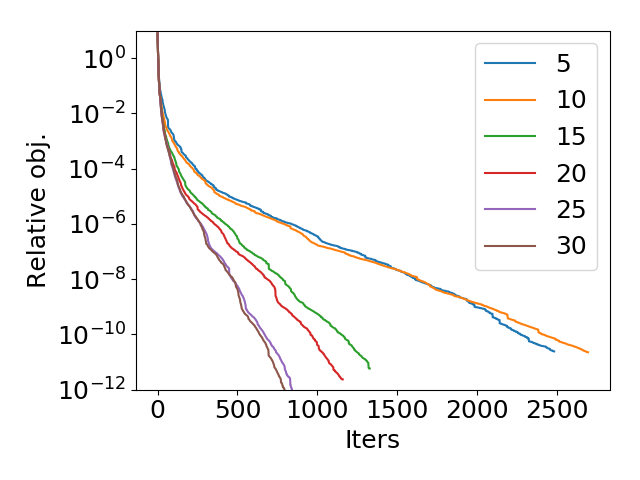}&
\includegraphics[width=.32\linewidth]{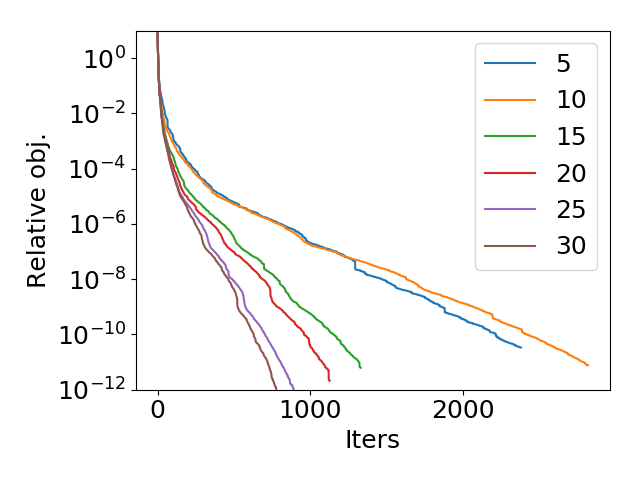}
\end{tabular}
\caption{rcv1t}
\end{subfigure}
\caption{Comparison of different subproblem solution exactness in
	solving \eqref{eq:l1lr}. The
	y-axis is the relative objective error \eqref{eq:optf}, and the
x-axis is the iteration count.}
\label{fig:conv}
\end{figure}

\begin{figure}[tb]
\centering
\begin{subfigure}[b]{0.32\linewidth}
\includegraphics[width=\linewidth]{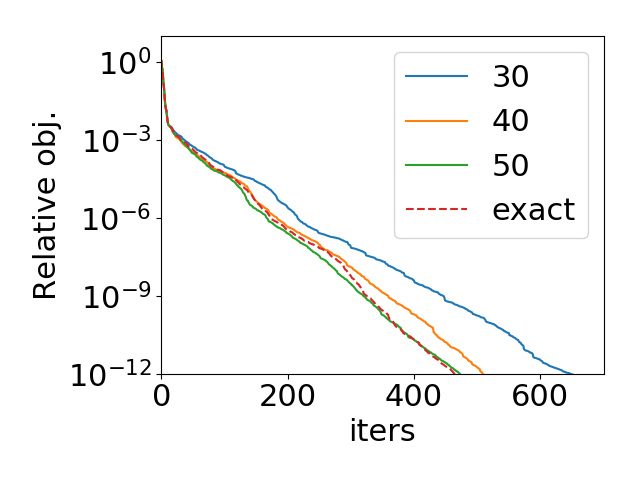}
\includegraphics[width=\linewidth]{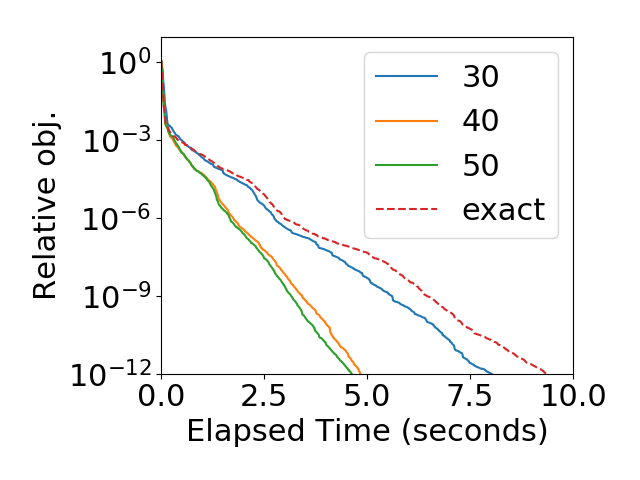}
\caption{a9a}
\end{subfigure}
\begin{subfigure}[b]{0.32\linewidth}
\includegraphics[width=\linewidth]{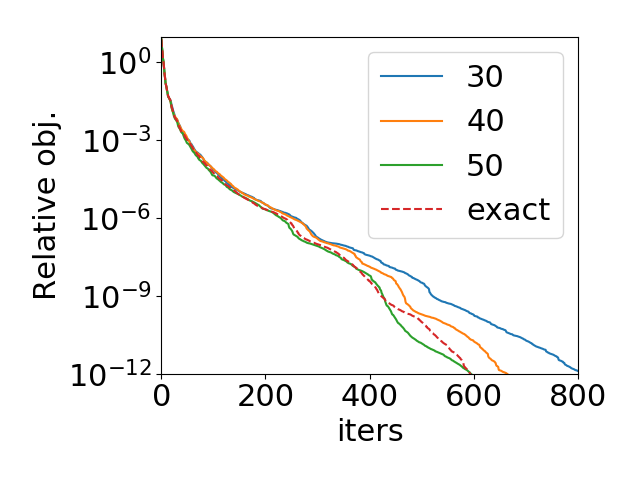}
\includegraphics[width=\linewidth]{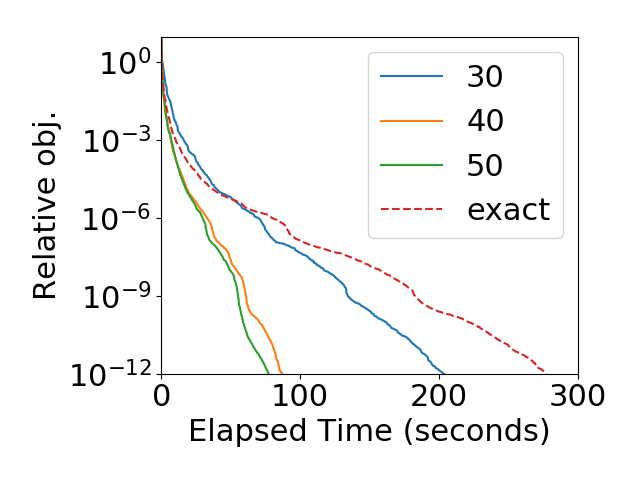}
\caption{rcv1}
\end{subfigure}
\begin{subfigure}[b]{0.32\linewidth}
\includegraphics[width=\linewidth]{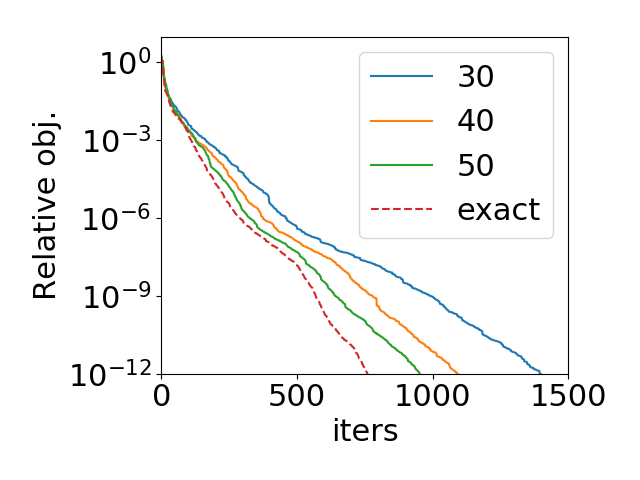}
\includegraphics[width=\linewidth]{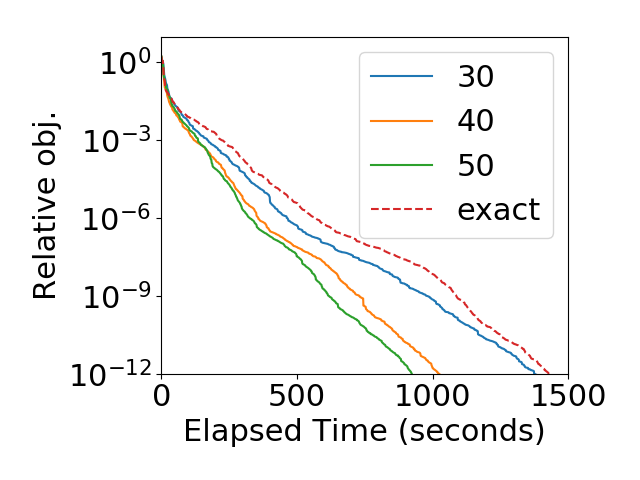}
\caption{epsilon}
\end{subfigure}
\caption{Comparison between the exact version and the
	inexact version of Algorithm \ref{alg:QuasiArmijo} for solving
	\eqref{eq:l1lr}.
Top: outer iterations; bottom: running time. The y-axis is the
relative objective error \eqref{eq:optf}.}
\label{fig:exact}
\end{figure}

\subsection{Dual of $\ell_2$-regularized Squared-Hinge Loss Minimization}

Given the same binary-labelled data points as in the previous
experiment and a parameter $C>0$, the $\ell_2$-regularized
squared-hinge loss minimization problem is
\begin{equation*}
	\min_{x\in \R^n}\, \frac12 \|x\|^2_2 + C \sum_{i=1}^l \max (1 -
	b_i a_i^T x, 0 )^2.
\end{equation*}
With the notation $A \coloneqq (b_1 a_1, b_2 a_2,\dotsc, b_l a_l)$, the dual
of this problem is
\begin{equation}	\label{eq:dual}
\min_{\alpha \ge 0} \, \frac{1}{2} \alpha^T A^T A \alpha - \bfone^T
\alpha + \frac{1}{4C} \|\alpha\|_2^2,
\end{equation}
which is $(1/ 2C)$-strongly convex.
We consider the distributed setting such that the columns of $A$ are
stored across multiple processors.  In this setup, only the
block-diagonal parts (up to a permutation) of $A^T A$ can be easily
formed locally on each processor. We take $H_k$ to be the matrix
formed by these diagonal blocks, so that the subproblem
\eqref{eq:quadratic} can be decomposed into independent parts. We use
cyclic coordinate descent with random permutation (RPCD) as the solver
for each subproblem. (Note that this algorithm partitions trivially
across processors, because of the block-diagonal structure of $H_k$.)

Our experiment compares the strategy of performing a fixed number of
RPCD iterations for each subproblem with one of increasing the number
of inner iterations as the algorithm proceeds, as in
\cite{SchT16a,GhaS16a}. We use the data sets in Table~\ref{tbl:data},
and compare the two strategies on Algorithm~\ref{alg:QuasiArmijo}, but
use an exact line search to choose $\alpha_k$ rather than the
backtracking approach. (An exact line search is made easy by the
quadratic objective.)  For the first strategy, we use ten iterations
of RPCD on each subproblem, while for the second strategy, we perform
$1 + \lfloor k / 10\rfloor$ iterations of RPCD at the $k$th outer
iteration as suggested by \cite{SchT16a,GhaS16a}. The implementation
is a modification of the experimental code of \cite{LeeR15a}.  We run
the algorithms on a local cluster with 16 machines, so that $H_k$
contains 16 diagonal blocks.  Results are shown in
Figure~\ref{fig:dual}. Since the choice of $H_k$ in this case does not
capture global curvature information adequately, the strategy of
increasing the accuracy of subproblem solution on later iterations
does not reduce the number of iterations as significantly as in the
previous experiment.  The runtime results show a significant advantage
for the first strategy of a fixed number of inner iterations,
particularly on the a9a and rcv1 data sets.
Judging from the trend in the approach of increasing inner iterations,
we can expect that the exact version will show huger disadvantage for
running time in this case.
We also observe the faster linear rate on early iterations, matching
our theory.

\begin{figure}[tb]
\centering
\begin{subfigure}[b]{0.32\linewidth}
\includegraphics[width=\linewidth]{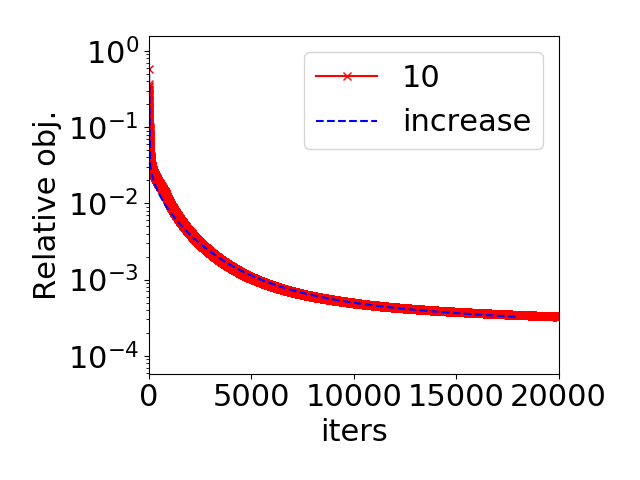}
\includegraphics[width=\linewidth]{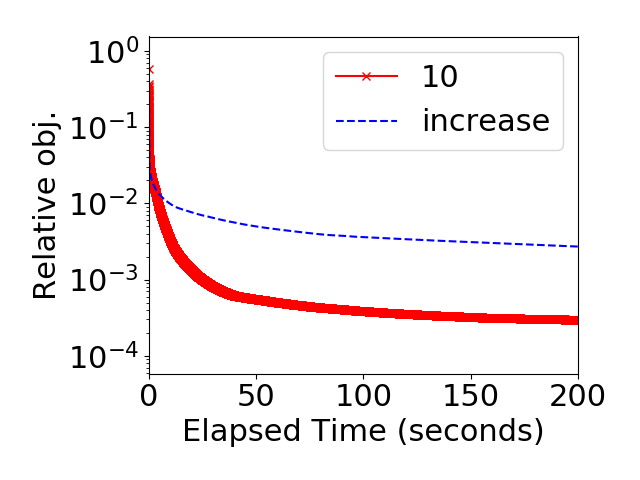}
\caption{a9a}
\end{subfigure}
\begin{subfigure}[b]{0.32\linewidth}
\includegraphics[width=\linewidth]{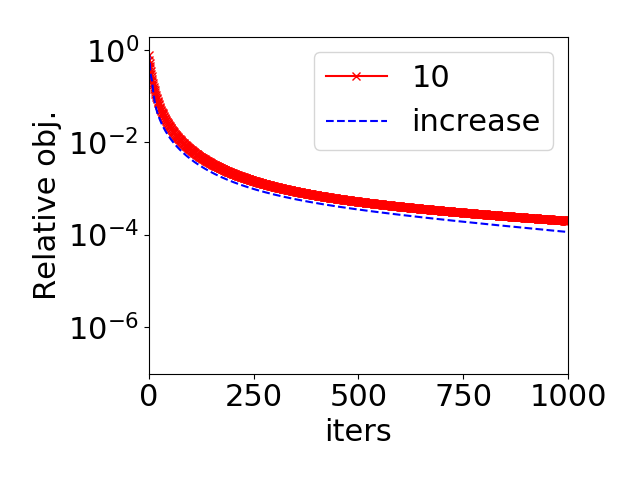}
\includegraphics[width=\linewidth]{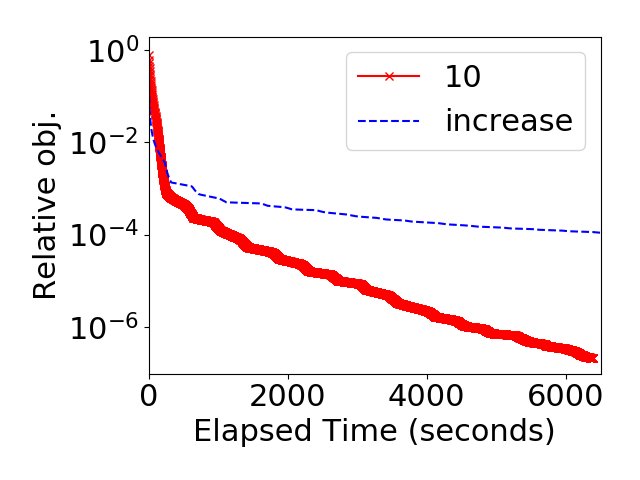}
\caption{rcv1}
\end{subfigure}
\begin{subfigure}[b]{0.32\linewidth}
\includegraphics[width=\linewidth]{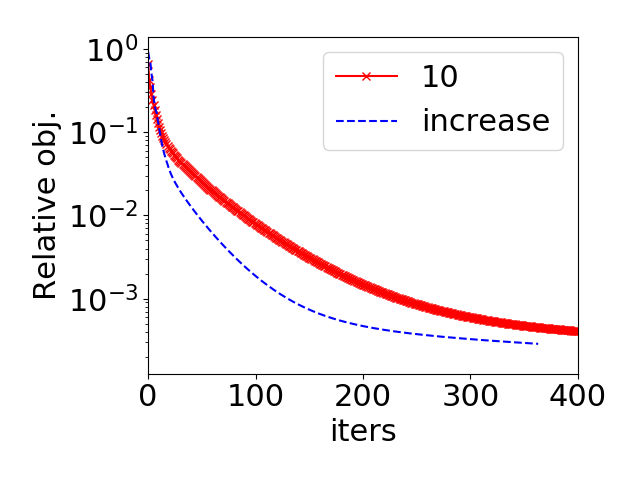}
\includegraphics[width=\linewidth]{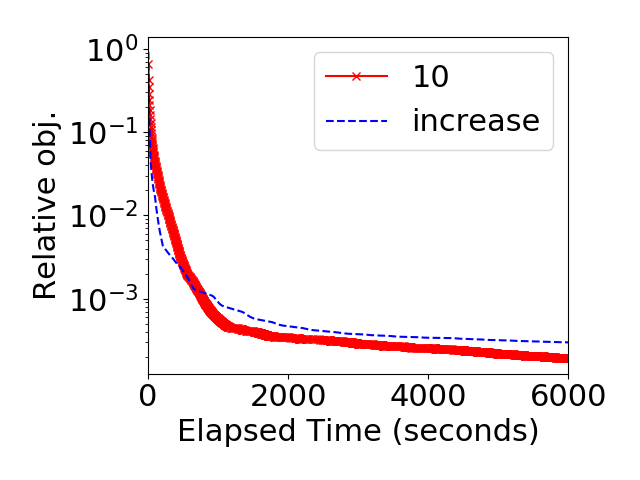}
\caption{epsilon}
\end{subfigure}
\caption{Comparison of two strategies for inner iteration count in
  Algorithm~\ref{alg:QuasiArmijo} applied to \eqref{eq:quadratic}:
  Increasing accuracy on later iterations (blue) and a fixed number of
  inner iterations (red).  Top: outer iterations; bottom: running
  time. Vertical axis shows relative objective error \eqref{eq:optf}.}
\label{fig:dual}
\end{figure}

\section{Conclusions} \label{sec:conclusions}
We have analyzed global convergence rates of three practical inexact
successive quadratic approximation algorithms under different
assumptions on the objective function, including the nonconvex case.
Our analysis shows that inexact solution of the subproblems affects
the rates of convergence in fairly benign ways, with a modest factor
appearing in the bounds. When linearly convergent methods are used to
solve the subproblems, the inexactness condition holds when a fixed
number of inner iterations is applied at each outer iteration $k$.



\bibliographystyle{spmpsci}      
\bibliography{inexactprox}

\appendix
\section{Proof of Lemma \ref{lemma:Q}}
\label{app:lemmaq}
\begin{proof}
  We have
  \begin{subequations}
\begin{align}
Q^* &= \min_{d} \, \nabla f\left(x\right)^T d + \frac{1}{2}d^T H d +
	\psi\left(x+ d\right) - \psi\left(x\right) \nonumber\\
\label{eq:derive1}
&\leq \min_d \, f\left(x+d\right) + \psi\left(x+d\right) +
\frac{1}{2}d^T H d - F\left(x\right) \\
\label{eq:derive1a}
&\leq F\left(x + \lambda
	\left( P_{\Omega}\left(x\right)  - x \right) \right)+
	\frac{\lambda^2}{2}\left(P_\Omega\left(x\right)  -
	x\right)^T H \left(P_\Omega\left(x\right) - x
	\right) - F\left(x\right) \;\;
        \forall \lambda \in [0,1] \\
\label{eq:derive2}
&\leq \left(1 - \lambda\right)
	F\left(x\right) + \lambda F^* - \frac{\mu \lambda \left(1
	-\lambda\right) }{2} \left\|x - P_\Omega\left(x\right)\right\|^2\\
	\nonumber
	&\qquad +
	\frac{\lambda^2 }{2} \left(x - P_\Omega\left(x\right)\right)^T
	H\left(x - P_\Omega\left(x\right)\right)  -
	F\left(x\right) \;\; \forall \lambda \in [0,1] \\
        \nonumber
&\leq 
	\lambda(F^* - F\left(x\right)) - \frac{\mu \lambda \left(1
	-\lambda\right) }{2} \left\|x - P_\Omega\left(x\right)\right\|^2 +
	\frac{\lambda^2 }{2} \|H\| \| x - P_\Omega\left(x\right)\|^2 \;\; \forall\lambda \in [0,1], 
        \nonumber
\end{align}
  \end{subequations}
where in \eqref{eq:derive1} we used the convexity of $f$, in \eqref{eq:derive1a} we set $d=\lambda(P_{\Omega}(x)-x)$,
and in \eqref{eq:derive2} we used the optimal set strong convexity
\eqref{eq:strong} of $F$. Thus we obtain \eqref{eq:Q}.
	\qed
\end{proof}

\section{Proof of Lemma \ref{lemma:seq}}
\label{app:lemmaseq}
\begin{proof}
Consider
\begin{equation}
	\lambda_k = \arg\min_{\lambda \in [0,1]} \, -\lambda \delta_k +
	\frac{\lambda^2}{2}A_k,
	\label{eq:lambdabound}
\end{equation}
then by setting the derivative to zero in \eqref{eq:lambdabound}, we have
\begin{equation}
\lambda_k = \min\left\{1, \frac{\delta_k}{A_k}\right\}.
\label{eq:lambdat}
\end{equation}
When $\delta_k \geq A_k$, we have from
\eqref{eq:lambdat} that $\lambda_k = 1$.
Therefore, from \eqref{eq:seq} we get
\begin{equation*}
\delta_{k+1} \leq \delta_{k} + c_k\left(- \delta_k + \frac{A_k}{2}\right)
\leq \delta_{k} + c_k\left(-\delta_k + \frac{\delta_k}{2}\right)
= \left(1 - \frac{c_k}{2}\right) \delta_{k},
\end{equation*}
proving \eqref{eq:linear}.

On the other hand,
since $A \geq A_k > 0, c_k \geq 0$ for all $k$,
\eqref{eq:seq} can be further upper-bounded by
\begin{equation*}
\delta_{k+1} \leq \delta_k +
	c_k\left(-\lambda_k \delta_k + \frac{A_k}{2}\lambda_k^2\right)
	\leq \delta_k + c_k \left( -\lambda_k \delta_k + \frac{A}{2}
\lambda_k^2 \right),\quad \forall \lambda_k \in [0,1].
\end{equation*}
Now take
\begin{equation}
\lambda_k = \min\left\{1, \frac{\delta_k}{A}\right\}.
\label{eq:lambdat2}
\end{equation}
For $\delta_k \geq A \geq A_k$, \eqref{eq:linear} still applies.
If $A > \delta_k$, we have from \eqref{eq:lambdat2}
that $\lambda_k = \delta_k / A$, hence
\begin{equation}
\delta_{k+1} \leq \delta_k - \frac{c_k}{2A}\delta_k^2.
\label{eq:tmp}
\end{equation}
This together with \eqref{eq:linear} imply that $\{\delta_k\}$ is a
monotonically decreasing sequence.
Dividing both sides of \eqref{eq:tmp} by $\delta_{k+1}\delta_k$, and
from the fact that $\delta_k$ is decreasing and nonnegative, we
conclude
\begin{equation*}
\delta_k^{-1} \leq \delta_{k+1}^{-1} -  \frac{c_k
	\delta_k}{2\delta_{k+1}A} \le \delta_{k+1}^{-1} -  \frac{c_k}{2A}
\end{equation*}
Summing this inequality from $k_0$, and using $\delta_{k_0} < A$, we obtain
\begin{equation*}
	\delta_k^{-1} \geq \delta_{k_0}^{-1} + \frac{\sum_{t=k_0}^{k-1}c_t}{2 A} \geq
	\frac{\sum_{t=k_0}^{k-1}c_t + 2}{2 A} \Rightarrow \delta_k \leq
	\frac{2 A}{\sum_{t=k_0}^{k-1} c_t + 2},
\end{equation*}
proving \eqref{eq:sub}.
	\qed
\end{proof}

\end{document}